%% file: arxiv.tex
\documentclass[11pt]{article}

\usepackage{setspace}
\usepackage[utf8]{inputenc}

\usepackage{amssymb,amsmath,amsthm,amsfonts}
\usepackage[margin=1truein]{geometry}

\usepackage{xcolor}
 \definecolor{bordeaux}{RGB}{100,0,50}
 \definecolor{darkblue}{RGB}{25, 25, 112}

\usepackage[pdftex,colorlinks=true,linkcolor=blue,citecolor=blue,urlcolor=red,unicode=true,hyperfootnotes=false,bookmarksnumbered]{hyperref}
\hypersetup{linkcolor={{blue!70!black}}, citecolor={black}, urlcolor={blue!70!black}}
\hypersetup{linkcolor=bordeaux, citecolor = darkblue, urlcolor = darkblue}

\input{macros}

\title{A logical approach to concentration}
\author{Michael Benedikt\thanks{Department of Computer Science, University of Oxford,
Oxford OX1 3QD, UK} \and Maksim Zhukovskii\thanks{School of Computer Science, The University of Sheffield, Sheffield S1 4DP, UK}}
\date{}

\begin{document}

\maketitle

\input{abstract}

\section{Introduction} \label{sec:intro}

For a sequence of real-valued random variables $\xi_n$, we say it \emph{concentrates around its expectation $\mathbb{E}\xi_n$} or simply \emph{concentrates}
if 
 $\xi_n/\mathbb{E} \xi_n$ converges in probability to $1$, as $n\to\infty$. 
This paper concerns the most common {\erdosrenyi} (or binomial) model of random graphs, denoted by $G(n,p)$, where each edge between a pair of vertices from an $n$-vertex set is present with probability $p=p(n)$, independently of the others. A function $\tau$ that maps graphs to reals, evaluated on $G(n,p)$, induces a sequence  of random variables, and  we say that $\tau$ concentrates if the random variables do. Many natural graph functions, such as counts of subgraphs, are known to concentrate in $G(n,p)$ for a broad range of $p=p(n)$: for example, the number of triangles concentrates for all $p$ such that $pn\to\infty$ as $n\to\infty$~\cite{smallsubgraphs} and the maximum degree concentrates for all $p$ such that $pn/\ln n\to\infty$ as $n\to\infty$~\cite{Bollobas_degrees,degreesequences,ivchenko}.\\ 

The concentration phemonemon was extensively studied for diverse graph characteristics and for various ranges of $p$. In particular, for $p$ above certain threshold, 
\begin{itemize}
\item for every graph $H$, the number of subgraphs $X_H$ isomorphic to $H$ in $G(n,p)$ concentrates~\cite{smallsubgraphs},
\item for every graph $H$ with a distinguished set of roots, the maximum number of ``$H$-extensions'' of a set of roots in $G(n,p)$ concentrates for $p$ above a certain threshold~\cite{Spencer_count} (for example, the maximum degree is the maximum number of $H$-extensions of a vertex, where $H$ has two vertices, a single edge, and one root).
\end{itemize}
There are many other natural directions around the concentration phenomenon that have been thoroughly investigated, such as tail bounds, integral and local limit theorems, see, e.g.~\cite{Araujo,Chatterjee,Samotij,Infamous,KimVu,smallsubgraphs,SS}.

In this paper, we introduce a broad logic-based approach to obtaining concentration results. We define a real-valued logic consisting of terms that represent real-valued graph functions. The language is closed under a wide class of real-valued
functions -- the real-valued analog of Boolean connectives. It is also closed under a summation aggregate as well as the minimum and maximum aggregates, the real-valued analog of existential and universal quantification. Our main result states that every term in this language concentrates when $p=\mathrm{const}$ or $p=n^{-\alpha}$, where $\alpha\in(0,1)$ is irrational. In addition, we show that, with certain restrictions on the connectives, the expectation of a term  is approximated by a polynomial and converges to a constant whenever the term is bounded. This result subsumes the previously mentioned concentration results for subgraph counts and maximum extensions. 
Therefore, our results can be seen as meta-theorems for concentration: they both imply a number of concentration results on random graphs in the literature, as well as serve for generating new concentration results. 

On the other hand, our results can be viewed within the long line of work on logic of random graphs. The bulk of these results deal with first-order (FO) logic and focus on convergence. 
 This line originates in works of Fagin \cite{faginzeroone} and Glebskii et al. \cite{glebskii}, which showed that for $p=\mathrm{const}$, the probability of any sentence in FO logic over graphs tends to either zero or one: that is, $G(n,p)$ obeys the \emph{ FO zero-one law}. 
 Later work established zero-one laws for other Boolean-valued logics for $p=\mathrm{const}$ \cite{kolaitisvardiinfty}, FO zero-one laws and convergence laws for $G(n,p)$ for various growth rates of $p=p(n)$
\cite{lynchlinearsparse,shelahspencersparse,spencerthoma}, for fragments of FO logic~\cite{McArthur,YZ,zhuk-dense,zhuk-sparse}, and for other random graph models 
\cite{comptonsurvey,haber,muller,kleinberg,attachment2,attachment,mccolm}. 
 The validity of the FO zero-one law for $G(n,p)$   depends on the value of $p$:
  \begin{itemize}
  \item it obeys the FO zero-one law when $p=n^{o(1)}$ (in particular, for constant $p$)~\cite{glebskii,faginzeroone,spencer-dense}; 
  \item for $p=n^{-\alpha}$, where $\alpha\in(0,1]$, it obeys the law if and only if $\alpha$ is irrational~\cite{shelahspencersparse,lynchlinearsparse}.
  \end{itemize}
  The case $\alpha>1$ is also well-studied in the context of FO limit laws~\cite{shelahspencersparse} and is not the focus of this paper: in this regime, with high probability the random graph is acyclic and, for every tree $T$, the asymptotic behaviour of the number of components isomorphic to $T$ is well-understood. For any rational $\alpha\in(0,1)$, there are FO sentences such that the probability of their validity does not converge (and therefore there is no hope to get a concentration result in our language as well). We stress that convergence of terms in real-valued logic has been also investigated~\cite{usneurips24,uslics,jaeger,koponen,koponen2024relative}. These results do not give any insight into concentration, for the detailed comparison of these results with our contribution see Section~\ref{sec:related}. 
  
  \paragraph{Organisation.} In Section~\ref{sec:prelims}, we introduce our language and recall the required graph-theoretic and probabilistic background. Section~\ref{sec:related} discusses related work and compares previous results with our contribution. Technical aspects of this comparison are disclosed later in Section~\ref{sc:comparison}. Classes of connective functions that preserve concentration and their properties are discussed in Section~\ref{sc:connectors}. These properties are further exploited in Sections~\ref{sec:dense-all}~and~\ref{sec:sparse-all}, where we state and prove our main results --- concentration of closed terms in dense and sparse random graphs, respectively. In Section~\ref{subsc:convergence-law-sparse}, we state another result that extends the convergence law from~\cite{uslics} to sparse $G(n,n^{-\alpha})$ with irrational $\alpha$. Section~\ref{sec:discussion} proposes related research directions and discusses some remaining open questions.

\paragraph{Notation and conventions.}
Given a positive integer $n$ we write $[n]$ for the set $\{1,\dots, n\}$. For positive integers $n$ and $k\leq n$ we denote $(n)_k:=\frac{n!}{(n-k)!}$.
For a set $X$ and a positive integer $k$, we write $(X)_k$ for the set of $k$-tuples of distinct elements from $X$. Therefore, if $X$ is finite, then $(|X|)_k$ is the cardinality of $(X)_k$. Over-lined variables,
such as $\overline{v}$, represent finite tuples of arbitrary length $|\overline{v}|$. Sometimes we  overload $\overline{v}$ to mean either the tuple or the underline {\it set} of elements that constitute the tuple.

\section{Preliminaries} \label{sec:prelims}

\paragraph{Graphs.} 
A graph $G$ consists of a finite set of vertices $V(G)$ and a set of undirected edges $E(G) \subseteq \binom{V(G)}{2}$ with no loops.
 For any vertex $u$, its neighborhood is $\Ne(u) \coloneqq \{u \in V(G) \mid \{u, v\} \in E(G)\}$.

\begin{definition}[Multi-rooted graphs $F{[\overline{x}]}$]
A \emph{multi-rooted graph}
(MRG) is a graph $F$ along with a tuple of distinguished vertices $\overline{x}=(x_1,\ldots,x_k)\in(V(F))_k$ called \emph{roots}. We let
$F[\overline{x}]$ denote such an MRG. We say that MRGs $F[\overline{x}]$ and $G[\overline{u}]$ are {\it isomorphic}, if $|\overline{x}|=|\overline{u}|=k$ for some positive integer $k$, and there exists an isomorphism $f:V(F)\to V(G)$ of (unrooted) graphs $F$ and $G$ such that, for every $i\in[k]$, $f(x_i)=u_i$.
\end{definition}

 Note that in this paper, all graphs, hence all MRGs, are finite.
The usual notions from graph theory are extended to MRGs in the natural way.

\begin{definition}[Induced subgraphs] \label{def:namingrestricting}
Given a graph $G$ and a set of its vertices $U\subset V(G)$, we let
$G \upharpoonright U$ denote the {\it induced subgraph} of $G$ on the set $U$, that is $G \upharpoonright U$ has vertex set $U$ and set of all edges of $G$ that have both endpoints in $U$.
\end{definition}

\paragraph{Probability theory background.}
In all the probabilistic statements and arguments, we always assume that there is some underline probability space that we do not define explicitly. When a random varible $X$ has distribution $Q$, we write $X\sim Q$.
\begin{definition}[With high probability] \label{def:aasconvergence}
We say that a sequence of events $(A_n)_{n\in \NN}$ holds \emph{with high probability (abbreviated to $\whp$)} if $\lim_{n\to \infty} \Pr(A_n)=1$. 
\end{definition}

\begin{definition}[Convergence] 
Let $(X_n)_{n \in \nats}$ be a sequence of random variables and $X$ be a single random variable. We say 
\begin{compactitem}
\item $(X_n)_{n \in \nats}$ \emph{converges in probability} to $X$ and write $X_n\stackrel{P}\to X$ if, for every $\varepsilon>0$, $\whp$ $|X_n-X|\leq\varepsilon$;
\item $(X_n)_{n \in \nats}$ \emph{converges in distribution} to $X$ and write $X_n\stackrel{d}\to X$ if the distribution of $X_n$ converges to the distribution of $X$ weakly, i.e., for every continuous bounded real-valued $f$, we have that $\mathbb{E}f(X_n)\to\mathbb{E}f(X)$;
\item $(X_n)_{n \in \nats}$ \emph{converges in $L_p$} to $X$ for some $p>0$ and write $X_n\stackrel{L_p}\to X$ if $\lim_{n\to\infty}\mathbb{E}|X_n-X|^p=0$.
\end{compactitem}
\end{definition}

We recall that that the convergence in $L_p$ is stronger than the convergence in probability which is, in turn, stronger than the convergence in distribution.

\begin{definition}[Concentration around the mean] \label{def:concentrates} We say that a sequence $(X_n)_{n \in \nats}$ of random variables \emph{concentrates around its mean} if, for every $p>0$, $X_n/\mathbb{E}X_n\stackrel{L_p}\to 1$ as $n\to\infty$.
\end{definition}

\paragraph{{\erdosrenyi} random graphs.}
Let $p=p(n)\in(0,1)$ be a function of $n\in\mathbb{Z}_{>0}$. The \erdosrenyi\ random graph $G(n,p)$ is a distrition on the set of all graphs on the set of vertices $[n]$ defined as follows: each edge between distinct $i, j \in [n]$ appears with probability $p(n)$, independently of the others. That is, if $\mathbf{G}_n\sim G(n,p)$, then for every graph $G$ on $[n]$, we have that $\mathbb{P}(\mathbf{G}_n=G)=p^{|E(G)|}(1-p)^{{n\choose 2}-|E(G)|}$. We single out a few cases of the growth rate of $p$ for {\erdosrenyi} random graphs, that are of particular interest in the context of logical limit laws:
\begin{compactitem}
\item \emph{dense}: $p$ is a constant;
\item \emph{irrational sparse}: (or just ``sparse'' , since it is the only sparse case we deal with) $p= n^{-\alpha}$, where $0<\alpha<1$ is irrational.
 \end{compactitem}

\paragraph{A real-valued logic with aggregate operators.}
 We present the real-valued logics we will analyze.
 Let $\functionclass$ be a collection of functions from $\reals^m$ into $\reals$, where $m$ is a positive integer. Following the literature on continuous logic, we  will refer to these as \emph{connective functions}, since they are analogous to connectives in classical logic.
 
\begin{definition}[Aggregate real-valued logic] The term language
    $\Agglang_{\functionclass}$ contains node variables $u, v, w, \ldots$ and terms defined inductively as follows.
    \begin{compactitem}
        \item The \emph{basic terms}, or \emph{atomic terms}, are the 
        constants $c$, the characteristic function of the edge relation $\mathrm E(u, v)$, and equality of vertices $u=v$.
        \item Given a term $\tau(\overline{u}, v)$ with free variables $\overline{u},v$, the \emph{global average} for the variable $v$ is
        $\Mean_v \tau(\overline{u},v)$ with free variables $\overline{u}$; and given a term $\tau(w,\overline{u}, v)$ with free variables $w,\overline{u},v$, the \emph{local average} for variables $w,v$ is $\LMean_{wEv} \tau(w,\overline{u},v)$ with free variables $w,\overline{u}$.
        \item Given 
        $\tau(\overline{u},v)$ with free variables $\overline{u},v$, the \emph{maximum} for variable $v$ is $\supl_v \tau(\overline{u},v)$ with free variables $\overline{u}$.
        \item Given 
        $\tau(\overline{u},v)$ with free variables $\overline{u},v$, the \emph{minimum} for variable $v$ is $\infl_v \tau(\overline{u},v)$ with free variables $\overline{u}$.
        \item Given 
        $\tau(\overline{u},v)$ with free variables $\overline{u},v$, the \emph{summation} for variable $u$ is $\Sigma_v \tau(\overline{u},v)$ with free variables $\overline{u}$.
        \item Terms are closed under applying a  function symbol for each  $f  \in \functionclass$.
    \end{compactitem}
\end{definition}

We define sublanguages, like $\Aggo{\functionclass, \LMean,\Mean, \supl}$, by omitting one or more of the aggregates.

\begin{definition}[Interpretation of terms]
    Let $\tau$ be a term with $k$ free variables, let $G$ be a graph, and let $\overline{u}=(u_1,\ldots,u_k)\in (V(G))^k$. Let $\GG = G[\overline{u}]$. For $v\in V(G)$, we let $\GG[v] = G[(u_1,\ldots,u_k,v)]$. The \emph{interpretation} $[\![\tau]\!]_{\GG}$ of a term $\tau$ in $\GG$ is defined recursively as follows:
    \begin{compactitem}
        \item $[\![c]\!]_{\GG} = c$ for any constant $c\in\mathbb{R}_{\geq 0}$.
        \item $[\![\mathrm E(u_{i}, u_{j})]\!]_{\GG}$ is $1$ when $\{u_i, u_j\} \in E(G)$ and $0$ otherwise, and similarly for equality.
        \item $[\![f(\tau_1, \ldots, \tau_m)]\!]_{\GG}  = f([\![\tau_1]\!]_{\GG}, \ldots, [\![\tau_m]\!]_{\GG})$ for $f \in \functionclass$.
        \item Define $[\![\Mean_v \tau]\!]_{\GG}$ as $
        \frac 1 {\abs{V(G)}}
                \sum_{v \in V(G)} [\![\tau]\!]_{\GG[v]}$.
     
        \item Define $[\![ \LMean_{u_1 E v} \tau]\!]_{\GG}$ as $
        \frac 1 {\abs{\Ne(u_1)}}
                \sum_{v \in \Ne(u_1) } [\![\tau]\!]_{\GG[v]}$,
        if the denominator is nonzero, and zero otherwise.
       
        \item $[\![\supl_v \tau]\!]_{\GG} = \max_{v \in V(G)} [\![\tau]\!]_{\GG[v]}$, and same for $\infl$.
        \item $[\![\Sigma_v \tau]\!]_{\GG} = \Sigma_{v \in V(G)} [\![\tau]\!]_{\GG[v]}$.
    \end{compactitem}
\end{definition}

A term is \emph{closed} if it has no free variables.

Throughout the paper, we sample $\mathbf{G}_n\sim G(n,p)$, fix a $k$-tuple $(u_1,\ldots,u_k)\in[n]^k$, and consider the random variable  $[\![\tau]\!]_{\mathbf{G}_n[\overline{u}]}$. Note that \emph{we will often overload $\tau$ to mean either the term or the random variable. When we say some property holds ``almost surely'' for $\tau$ (or simply ``a.s.''), 
we will mean that it holds with probability one as a random variable. For the probability spaces we consider, this will follow from the fact that the property holds always for $\tau$ as a function of a graph and a tuple.} In particular, we say that $\tau$ converges in probability/in distribution/in $L_p$ if the corresponding sequence of random variables $([\![\tau]\!]_{\mathbf{G}_n[\overline{u}]})_{n\in\mathbb{Z}_{>0}}$ converges in the same sense. Similarly we say that $\tau$ \emph{concentrates around the mean} if the corresponding sequence of random variables concentrates in the sense of Definition \ref{def:concentrates}.

\paragraph{Examples of the term language.}
Our language is quite expressive. 
\begin{compactitem}
\item All of the function classes considered in ths paper will allow us to simulate binary Boolean connectives: that is, for every $f_0:\in\{0,1\}^2\to\{0,1\}$, there exists $f\in\mathcal{F}$ such that, for all $x\in\{0,1\}^2$, $f(x)=f_0(x)$. Then, for any formula $\phi$ in FO logic, there is a term $\tau_\phi\in\Agglang_{\mathcal{F}}$ that returns $1$ when $\phi$ holds and $0$ otherwise.
We form $\tau_\phi$ inductively, applying maximum to simulate existential quantifications, and using connective functions to simulate  Boolean connectives.

\item For any FO formula $\phi$, there is a term $\tau^{\%}_{\phi}$ in the sublanguage $\Aggo{\mathcal{F},\Mean,\supl}$  that returns the {\it percentage} of tuples of nodes
satisfying $\phi$. If $\phi$ has a single variable, we can  write a term $\tau^\%_{ \phi}(v)\in\Aggo{\mathcal{F},\LMean,\supl}$ that returns
the percentage of $v$'s neighbors that satisfy~$\phi$.
\item For any FO formula $\phi$, there is a term $\tau^{\mathrm{count}}_{\phi}$ in the sublanguage $\Aggo{\mathcal{F},\Sigma,\supl}$  that returns the {\it number} of tuples of nodes
satisfying $\phi$. For example, there is a term that counts the number of triangles.
\end{compactitem}

In Section~\ref{sc:connectors}, we will define classes of functions $\functionclassrellip$ and $\functionclasspoly$, and show that $\Agglang_{\mathcal{F}}=\Aggo{\mathcal{F},\Sigma,\supl}$ for both $\mathcal{F}=\functionclassrellip$ and $\mathcal{F}=\functionclasspoly$.

\section{Related work} \label{sec:related}

In the dense case of $p=\mathrm{const}$, \cite{faginzeroone,glebskii} showed that the probability of each FO sentence approaches either $0$ or $1$ as $n\to\infty$. For terms in $\Aggo{\Lipsch,\Mean,\LMean}$, where $\Lipsch$ is the set of Lipschitz functions  \cite{usneurips24} showed that every term converges in probability to a constant. \cite{uslics} extended this to include $\supl$.  Our concentration results imply these convergence results for bounded terms in the language $\mathrm{Agg}_{\functionclasspoly}$ with summation, although for a different class of connectives, and if it is 0-1 valued it implies a zero or one limit (see details in Section~\ref{sc:comparison}) --- in particular, our result subsumes the FO 0-1 law. Note that these convergence results cannot be extended to bounded terms in the entire $\mathrm{Agg}_{\Lipsch}$: for example, $\sin x\in\Lipsch$ and so $\sin|E(G)|$ is expressible in $\mathrm{Agg}_{\Lipsch}$, while $\sin|E(\mathbf{G}_n)|$ does not converge for $\mathbf{G}_n\sim G(n,p)$.  This happens for functions $f$ such that there exists a sequence of random variables $\xi_n$ that is concentrated around the mean, while $f(\xi_n)$ is not concentrated. As we show in this paper, this is the only obstacle for a language with summation to have concentrated terms.

For sparse random graphs $\mathbf{G}_n\sim G(n,n^{-\alpha})$, where $0<\alpha<1$,
\cite{shelahspencersparse} showed the FO 0-1 law if $\alpha$ is irrational. For rational $\alpha\in(0,1)$, they proved that there exists a FO sentence $\phi$ such that $\mathbb{P}(\mathbf{G}_n\models\phi)$ does not converge. It follows that we \emph{cannot have} convergence for any sublanguage of $\mathrm{Agg}_{\mathcal{F}}$ that includes $\supl$ when $\mathcal{F}$ contains all boolean connectives, in the case of rational $\alpha\in(0,1)$. In contrast, for the most restrictive language $\Aggo{\Lipsch,\Mean,\LMean}$ and for {\it all} $\alpha\in(0,1)$, \cite{usneurips24} showed convergence in probability for every term. We extend the results from \cite{usneurips24,uslics} by showing that, for irrational $\alpha\in(0,1)$, every term in the language $\Aggo{\Lipsch,\Mean,\LMean,\supl}$ converges in probability to a constant. We further  prove that every term in $\mathrm{Agg}_{\functionclassrellip}$ is concentrated and every bounded term in $\mathrm{Agg}_{\functionclasspoly}$ converges in probability to a constant. This result clearly implies the 0-1 law from~\cite{shelahspencersparse}. See a detailed comparison of these results in Section~\ref{sc:comparison}.

Convergence laws in the context of real-valued logics were also studied in~\cite{jaeger,koponen,koponen2024relative}.  The main theorem of~\cite{koponen} is an almost-sure aggregation elimination result  for a term language called PLA, which implies the convergence law for a wide family of distributions over graphs, including dense $G(n,p)$, generalising results from~\cite{jaeger}. In~\cite{koponen2024relative}, another aggregation elimination result for a wider family of probability distributions, which also subsumes sparse $G(n,p)$, was presented. However, the logic involved does not extend FO logic.  Crucially, terms in the languages considered in these papers take values in $[0,1]$ and aggregation functions include $\Mean$ and $\LMean$ but not $\Sigma$. 
 A general approach to asymptotic elimination of aggregation functions of $[0,1]$-valued terms, for arbitrary \emph{bounded} aggregation and connective functions, was presented in~\cite{koponen-general}.

Our results imply concentration of subgraph counts and extension counts, that were previously established in~\cite{smallsubgraphs,Spencer_count}. Nevertheless, we note that our methods can be used to establish fairly optimal bounds on the second order term in asymptotical expansion of these statistics --- for arbitrary extensions such results were not known, although for some specific MRGs optimal results were establised in~\cite{Zhuk-ext1,Zhuk-ext2,Zhuk-ext3,Warnke}. For example, an interesting and important case of extension counts is the maximum degree $\Delta_n$ which satisfies $\Delta_n=np+\sqrt{2np(1-p)\ln n}(1-o(1))$ $\whp$ when $p=\mathrm{const}$~\cite{Bollobas_degrees,degreesequences,ivchenko}, while our methods show the existence of a constant $C>0$ such that $\whp$ $|\Delta_n-np|\leq C\sqrt{n\ln n}$. For the number of subgraphs $X_H$ isomorphic to a given fixed graph $H$, the second-order asyptmotical term is well known: for $p$ above a certain threshold, $X_H=\mathbb{E}X_H+\eta_n\sqrt{\mathrm{Var}X_H}$, where $\eta_n$ converges in distribution to a standard normal random variable~\cite{smallsubgraphs}. We emphasise that the strength of our result lies in its generality rather than its optimality: rather than optimising asymptotic bounds for specific statistics, we develop a general framework that yields new concentration results and offers a new logical perspective on concentration phenomena in probability theory.

\section{Connective functions}
\label{sc:connectors}

In what follows, we define a class of real-valued functions $f$ that preserve the concentration property of a sequence of random variables, i.e. if sequences $\tau^{(i)}_n$ are concentrated, then $f(\tau_n)$, where $\tau_n=(\tau^{(1)}_n,\ldots,\tau^{(m)}_n)$, is concetrated as well. Clearly, $f$ preserves the concentration property for these sequences, if
$$
 \frac{|f(\tau)-\mathbb{E} f(\tau)|}{\mathbb{E}f(\tau)}=O\left(\frac{\|\tau-\mathbb{E}\tau\|}{\|\mathbb{E}\tau\|}\right).
$$
This motivates the following definition of relative Lipschitz functions.  

\begin{definition}[Relative Lipschitz functions]
We call  $f:\mathbb{R}^m_{>0}\to\mathbb{R}_{>0}$ {\it relative Lipschitz}, if there exists $C>0$ such that,
\begin{equation}
\text{
$\frac{|f(x)-f(y)|}{f(x)+f(y)}\leq C\cdot\frac{\|x-y\|}{\|x+y\|}\quad$ for all $x,y\in\mathbb{R}^m_{>0}$.
}
\label{eq:asymp_Lip}
\end{equation}

As we will use these functions in a context where arguments could  take value 0, we have to extend the definition to all non-negative $x$. For $f:\mathbb{R}^m_{\geq 0}\to\mathbb{R}_{\geq 0}$ and $I=\{j_1,\ldots,j_{m'}\}\subseteq[m]$ of size $m'\in[m]$, we define the {\it $I$-specialisation of $f$} as a function $f_I:\mathbb{R}^{m'}_{>0}\to\mathbb{R}_{\geq 0}$ satisfying the following requirement. 
For every $y=(y_1,\ldots,y_{m'})\in\mathbb{R}_{>0}^{m'}$ and its `extension' $x=(x_1,\ldots,x_m)$, where
$x_{j_i}=y_i>0$ for $i\in[m']\text{ and }x_j=0\text{ for }j\notin I$,
$$
f_I(y)=f(x).
$$ 
A function $f:\mathbb{R}^m_{\geq 0}\to\mathbb{R}_{\geq 0}$ is {\it relative Lipschitz} if every specialisation is either positive-valued and relative Lipschitz, or identically equal to zero.
\end{definition}

\begin{remark}
\label{rk:Hillbert}
For the 1-dimensional case $f:\mathbb{R}_{>0}\to\mathbb{R}_{>0}$, the class of relative Lipschitz functions coincides with the class of Lipschitz functions on $(\mathbb{R}_{>0},\rho)$, where the metric $\rho$ is a subadditive transform of the Hilbert metric $d(x,y)=|\ln(x/y)|$. Indeed, $|x-y|/(x+y)=g(d(x,y))$, where $g(t)=\frac{e^t-1}{e^t+1}$ is a subadditive function. Therefore, $\rho$ is indeed a metric. However, when the dimension $m>1$, the $\|x-y\|/\|x+y\|$ is not a metric.
\end{remark}

Let $\functionclassrellip$ 
be the family of relative Lipschitz functions $f:\mathbb{R}^m_{\geq 0}\to\mathbb{R}_{\geq 0}$. In Section~\ref{sc:dense-proofs}, we show that, for the set of connective functions $\functionclassrellip$, every term $\tau$ in the respective language evaluated on $G(n,p)$, is conentrated around its average $\mathbb{E}\tau$. We will rely on the following theorem that shows that $f$ preserves concentration properties of random variables.

\begin{theorem}[Functions preserve concentration]
\label{cl:function-concentration}
Let $f:\mathbb{R}^m_{\geq 0}\to\mathbb{R}_{\geq 0}\in\functionclassrellip$. There exists constants $a_f,b_f>1$ such that the following holds. Let $\xi_1,\ldots,\xi_m\geq 0$ be random variables, $\varphi_1,\ldots,\varphi_m\geq 0$, $\delta,\delta'>0$, and $K>2$ be constants such that, for every $i\in[m]$, 
\begin{compactitem}
\item $\xi_i\in\{0\}\cup[K^{-1},K]$ a.s.,

\item either $\varphi_i=0$ and $\mathbb{P}(\xi_i=0)\geq 1-\delta$, 

or $\varphi_i\geq K^{-1}$ and $\mathbb{P}(|\xi_i-\varphi_i|>\delta'\cdot\varphi_i)\leq \delta$.
\end{compactitem}
Let $\xi=f(\xi_1,\ldots,\xi_m)$ and $\varphi=f(\varphi_1,\ldots,\varphi_m)$. Then
\begin{compactitem}
\item $\xi\in\{0\}\cup[K^{-b_f},K^{b_f}]$ a.s.,
\item either $\varphi=0$ and $\mathbb{P}(\xi=0)\geq 1-m\delta$, 

or $\varphi\geq K^{-b_f}$ and $\mathbb{P}(|\xi-\varphi|>a_f\delta'\cdot\varphi)\leq m\delta$.
\end{compactitem}
Moreover $\varphi=0$ if and only if $f_I\equiv 0$ where $I$ is the set of all $i\in[m]$ such that $\varphi_i\neq 0.$
\end{theorem}

\begin{proof}
By Claim~\ref{cl:poly-bound}, there exists $b_f>1$ that only depends on $f$, such that $\xi<K^{b_f}$ a.s., as needed. Moreover, whenever $f_I\not\equiv 0$, we have that $1/f_I\in\functionclassrellip$. Therefore, by Claim~\ref{cl:poly-bound}, whenever $\xi\neq 0$, we have that $1/\xi<K^{b_f}$, and the same applies to $\varphi$.

Let  $I=\{i_1,\ldots,i_{m'}\}\subset[m]$ be the set of  $i\in[m]$ such that $\varphi_i\neq 0$. In particular, for $i\notin I$, $\mathbb{P}(\xi_i=\varphi_i=0)\geq 1-\delta$, while, for $i\in I$, $\varphi_i\geq 1/K$ and $\mathbb{P}(|\xi_i-\varphi_i|>\delta'\cdot\varphi_i)\leq \delta$.

If $I=\varnothing$, we get $\varphi=f(0,\ldots,0)$ and $\tau=\varphi$ with probability at least $1-m\delta$, completing the proof (for any choice of $a_f>0$).

Now, let $I\neq\varnothing$ and $1\leq |I|=:m'\leq m$. If $\varphi=0$, then $f_I\equiv 0$, meaning that $\mathbb{P}(\xi=0)\geq 1-m\delta$.

Assume $\varphi\neq 0$ (therefore, $\varphi\geq K^{-b_f}$). It remains to show $\mathbb{P}(|\xi-\varphi|>a_f\delta'\cdot\varphi)\leq m\delta$.

Denote $\xi_I:=(\xi_{i_1},\ldots,\xi_{i_{m'}})$ and $\varphi_I:=(\varphi_{i_1},\ldots,\varphi_{i_{m'}})$. Let $C$ be the constant witnessing the relative Lipschitz condition~\eqref{eq:asymp_Lip} for $f_I$. Subject to the event 
$$
\{|\xi_i-\varphi_i|>\delta'\cdot\varphi_i\text{ for all }i\in I\}\subseteq
\{\xi_i\neq 0\text{ for all }i\in I\},
$$ 
that happens with probability at least $1-\delta m'$, we have
$$
 |f_I(\xi_I)-f_I(\varphi_I)|\leq|f_I(\xi_I)+f_I(\varphi_I)|\cdot C\cdot\frac{\|\xi_I-\varphi_I\|}{\|\xi_I+\varphi_I\|}.
$$
We also have that $\varphi=f_I(\varphi_I)$ and that $\xi=f_I(\xi_I)$ with probability at least $1-(m-m')\delta$. Therefore, with probability at least $1-(m-m')\delta-m'\delta= 1-m\delta$,
\begin{align*}
 |\xi-\varphi|\leq|\xi+\varphi|\cdot C\cdot \frac{\|\xi_I-\varphi_I\|}{\|\xi_I+\varphi_I\|}
 &\leq|\xi+\varphi|\cdot C\cdot\sqrt{\sum_{i\in I}(1-\xi_i/\varphi_i)^2}\\
 &\leq|\xi+\varphi|\cdot C\cdot\sqrt{\sum_{i\in I}(\delta')^2}
 =\delta' C\sqrt{m'}
\cdot|\xi+\varphi|,
\end{align*}
implying $|\xi-\varphi|\leq 2\delta' C\sqrt{m}\cdot \varphi$. Since the factor $2C\sqrt{m}$ only depends on $f$, we may take $a_f\geq 2C\sqrt{m}$, completing the proof.
\end{proof}

\paragraph{Properties of relative Lipschitz functions.}

We first note that, as in the case of Lipschitz functions, the notion of relative Lipschitz is stronger than continuity, when a function is defined on a positive domain.

\begin{claim} \label{cl:continuous}
If $f:\mathbb{R}^m_{>0}\to\mathbb{R}_{>0}$ is relative Lipschitz, then $f$ is continuous.
\end{claim}

\begin{proof}
Consider a sequence $x_n\to x$ in $\mathbb{R}^m_{>0}$. If there is a subsequence $x_{n_i}$ such that $f(x_{n_i})\to\infty$, then $|f(x)-f(x_{n_i})|/(f(x)+f(x_{n_i}))\to 1$, contradicting~\eqref{eq:asymp_Lip}. Therefore $\sup_n f(x_n)\leq C_0$ for some $C_0>0$. We get 
$$
 |f(x)-f(x_n)|\leq\left(C\cdot\frac{f(x)+f(x_n)}{\|x+x_n\|}\right)\|x-x_n\|\leq
 \frac{2CC_0\|x-x_n\|}{\|x+x_n\|}
 \to 0,\quad n\to\infty.
$$
\end{proof}

We claim that the family of relative Lipschitz functions is closed under addition, multiplication, and raising into a power.

\begin{claim}  \label{cl:sum-dif}
Let $f,g:\mathbb{R}^m_{>0}\to\mathbb{R}_{>0}$ be relative Lipschitz. Then
\begin{itemize}
\item $f+g$ is relative Lipschitz;
\item if, for some constant $\varepsilon>0$, $f\geq(1+\varepsilon)g$, then $f-g$ is relative Lipschitz.
\end{itemize}

\end{claim}

\begin{proof}
 Let $x\neq y\in\mathbb{R}^m_{>0}$. We have that $f(x)>0$, $f(y)>0$, $g(x)>0$, and $g(y)>0$. Then $(f+g)(x)$ and $(f+g)(y)$ are also positive. Morefore, if $f\geq(1+\varepsilon)g$, then $(f-g)(x)\geq\varepsilon g(x)>0$ and $(f-g)(y)\geq\varepsilon g(y)>0$. Then
\begin{align*}
 \left|(f+g)(x)-(f+g)(y)\right|&=
 |f(x)+g(x)-f(y)-g(y)|\leq
 |f(x)-f(y)|+|g(x)-g(y)|\\
 &\leq \left(C_f\cdot\frac{f(x)+f(y)}{\|x+y\|}+C_g\cdot\frac{g(x)+g(y)}{\|x+y\|}\right)\|x-y\|\\
 &\leq
 \max\{C_f,C_g\}\cdot\frac{(f+g)(x)+(f+g)(y)}{\|x+y\|}\cdot\|x-y\|.
\end{align*}
In the same way, since $f-g\geq\varepsilon g$, letting $C=\max\{C_f,C_g\}$, we get 
\begin{align*}
 \left|(f-g)(x)-(f-g)(y)\right|&\leq
 C\cdot\frac{(f+g)(x)+(f+g)(y)}{\|x+y\|}\cdot\|x-y\|\\
 &\leq
 C\cdot\frac{f(x)-g(x)+f(y)-g(y)+\frac{2}{\varepsilon}(f(x)-g(x)+f(y)-g(y))}{\|x+y\|}\cdot\|x-y\|\\
 &=(1+2/\varepsilon)C\cdot\frac{(f-g)(x)+(f-g)(y)}{\|x+y\|}\cdot\|x-y\|,
\end{align*}
completing the proof.
\end{proof}

\begin{claim} \label{clm:rellipproduct}
If $f,g:\mathbb{R}^m_{>0}\to\mathbb{R}_{>0}$ are relative Lipschitz functions then $f\cdot g$ is relative Lipschitz. 
In particular,  if $a>0$, then 
$a\cdot f$ is relative Lipschitz. 
\end{claim}

\begin{proof}
Let $x\neq y\in\mathbb{R}^m_{>0}$. We get that $fg$ has positive values and 
\begin{align}
 \left|(fg)(x)-(fg)(y)\right|&=
 |f(x)g(x)-f(x)g(y)+f(x)g(y)-f(y)g(y)|\notag\\
 &\leq
 f(x)|g(x)-g(y)|+g(y)|f(x)-f(y)|\notag\\
 &\leq \left(f(x)\cdot C_g\cdot\frac{g(x)+g(y)}{\|x+y\|}+g(y)\cdot C_f\cdot\frac{f(x)+f(y)}{\|x+y\|}\right)\|x-y\|.
\label{eq:product_1}
\end{align}
In the same way,
\begin{align}
 \left|(fg)(x)-(fg)(y)\right|\leq\left(f(y)\cdot C_g\cdot\frac{g(x)+g(y)}{\|x+y\|}+g(x)\cdot C_f\cdot\frac{f(x)+f(y)}{\|x+y\|}\right)\|x-y\|.
\label{eq:product_2}
\end{align}
If $g(y)\leq g(x)$, then~\eqref{eq:product_1} implies
\begin{align}
 \left|(fg)(x)-(fg)(y)\right|\leq\frac{(2C_g+C_f)f(x)g(x)+C_f f(y)g(y)}{\|x+y\|}\|x-y\|\leq C\frac{(fg)(x)+(fg)(y)}{\|x+y\|}\|x-y\|,
\label{eq:product-conclusion}
\end{align}
where $C=2C_g+C_f$. If $g(y)>g(x)$, then~\eqref{eq:product_2} implies the conclusion of~\eqref{eq:product-conclusion} with $C=2C_f+C_g$ as well, completing the proof.
\end{proof}

\begin{claim} \label{cl:all-powers}
If $f:\mathbb{R}^m_{>0}\to\mathbb{R}_{>0}$ is relative Lipschitz and $b\in\mathbb{R}$ is a constant, then $f^b$ is relative Lipschitz.
\end{claim}

\begin{proof}
Assume $b\neq 0$ --- otherwise the claim is obvious. Since $f>0$, we have that $f^b>0$ as well. Let $x\neq y\in\mathbb{R}^m_{>0}$. Without loss of generality, assume $f(x)\geq f(y)$.

Let $b>0$. We first observe that 
$$
f^b(x)-f^b(y)\leq \frac{2(f(x)-f(y))}{f^{1-b}(x)+f^{1-b}(y)}.
$$
Indeed,
$(f^b(x)-f^b(y))(f^{1-b}(x)+f^{1-b}(y))=f(x)-f(y)+f^b(x)f^{1-b}(y)-f^b(y)f^{1-b}(x)$ and the function 
$$
g(b):=f^b(x)f^{1-b}(y)-f^b(y)f^{1-b}(x)=f(y)(f(x)/f(y))^b-f(x)(f(y)/f(x))^b
$$ 
has 
$$
g'(b)=f(y)\ln(f(x)/f(y))(f(x)/f(y))^b+f(x)\ln(f(x)/f(y))(f(x)/f(y))^b\geq 0,
$$
implying $\max_{b\in[0,1]}g(b)=g(1)=f(x)-f(y)$.

Then
\begin{align*}
 f^b(x)-f^b(y)&\leq\frac{2(f(x)+f(y))}{f^{1-b}(x)+f^{1-b}(y)}\cdot C\cdot\frac{\|x-y\|}{\|x+y\|}\leq 2C\cdot\frac{f^b(x)+f^b(y)}{\|x+y\|}\|x-y\|,
\end{align*}
as needed.

Finally, let $b<0$. We get
\begin{align*}
  |f^b(x)-f^b(y)| =f^b(y)-f^b(x) =\frac{f^{-b}(x)-f^{-b}(y)}{f^{-b}(x)f^{-b}(y)}&\leq
  2C\cdot\frac{f^{-b}(x)+f^{-b}(y)}{f^{-b}(x)f^{-b}(y)}\cdot\frac{\|x-y\|}{\|x+y\|}\\
 &=2C\cdot\frac{f^{b}(x)+f^{b}(y)}{\|x+y\|}\cdot \|x-y\|,
\end{align*}
completing the proof.
\end{proof}

\begin{claim} \label{cl:min-max}
Let $f,g:\mathbb{R}^m_{>0}\to\mathbb{R}_{>0}$ be relative Lipschitz functions. Then $\max\{f,g\}$ and $\min\{f,g\}$ are relative Lipschitz.
\end{claim}

\begin{proof}
We have that $\max\{f,g\},\min\{f,g\}:\mathbb{R}^m_{>0}\to\mathbb{R}_{>0}$.
Let $x\neq y\in\mathbb{R}^m_{>0}$.

If $\max\{f(x),g(x)\}$ and $\max\{f(y),g(y)\}$ are achieved at the same function (say, $f$), then
\begin{align*}
 \left|\max\{f(x),g(x)\}-\max\{f(y),g(y)\}\right|
 &=|f(x)-f(y)|\leq C_f(f(x)+f(y))\frac{\|x-y\|}{\|x+y\|}\\
    &=C_f(\max\{f(x),g(x)\}+\max\{f(y),g(y)\})\frac{\|x-y\|}{\|x+y\|}.
\end{align*}
Without loss of generality, assume $\max\{f(x),g(x)\}=f(x)$ and $\max\{f(y),g(y)\}=g(y)$. Then
\begin{align*}
 \left|\max\{f(x),g(x)\}-\max\{f(y),g(y)\}\right|
 &=|f(x)-g(y)|\leq\max\{|f(x)-f(y)|,|g(x)-g(y)|\}\\
 &\leq\max\{C_f(f(x)+f(y)),C_g(g(x)+g(y))\}\frac{\|x-y\|}{\|x+y\|}\\
  &\leq 2(C_f+C_g)\max\{f(x),f(y),g(x),g(y)\}\frac{\|x-y\|}{\|x+y\|}\\
    &\leq 2(C_f+C_g)(\max\{f(x),g(x)\}+\max\{f(y),g(y)\})\frac{\|x-y\|}{\|x+y\|},\\
\end{align*}
as desired. 

Since $\min\{f,g\}=\frac{1}{\max\{1/f,1/g\}}$, we get the desired assertion for $\min\{f,g\}$ by applying Claim~\ref{cl:all-powers} with $b=-1$ twice.

\end{proof}

All the above claims can be also derived from the following general fact that the class $\functionclassrellip$ is composition-preserving.

\begin{claim}\label{cl:compose}
If $f:\mathbb{R}^k_{>0}\to\mathbb{R}_{>0}$ is relative Lipschitz and $f_1,\ldots,f_k:\mathbb{R}^m_{>0}\to\mathbb{R}_{>0}$ are relative Lipschitz, then $g:=f(f_1,\ldots,f_k):\mathbb{R}^m_{>0}\to\mathbb{R}_{>0}$ is relative Lipschitz as well.
\end{claim}

\begin{proof}
Let $x\neq y\in\mathbb{R}^m_{>0}.$ Then
\begin{align*}
 |g(x)-g(y)|&\leq C_f(g(x)+g(y))\cdot\frac{\|(f_1(x)-f_1(y),\ldots,f_k(x)-f_k(y))\|}{\|(f_1(x)+f_1(y),\ldots,f_k(x)+f_k(y))\|}\\
 &\leq C_f(g(x)+g(y))\cdot\frac{\|(C_1(f_1(x)+f_1(y)),\ldots,C_k(f_k(x)+f_k(y)))\|\cdot\|x-y\|/\|x+y\|}{\|(f_1(x)+f_1(y),\ldots,f_k(x)+f_k(y))\|}\\
 &\leq C_f(C_1+\ldots+C_k)(g(x)+g(y))\cdot\frac{\|x-y\|}{\|x+y\|},
\end{align*}
completing the proof.
\end{proof}

Note that, if $k=m=1$, then Claim~\ref{cl:compose} immediately follows from the equivalence between relative Lipschitz and Lipschitz in the transformed Hilbert metric, see Remark~\ref{rk:Hillbert}.

The results above 
show that the class of relative Lipschitz functions is fairly rich: together with the fact that $f(x_1,\ldots,x_m)=x_i$ is relative Lipschitz, which follows directly from the definition, they imply

\begin{corollary}
The following are relative Lipschitz
\begin{compactitem}
\item polynomials $f:\mathbb{R}^m_{\geq 0}\to\mathbb{R}_{\geq 0}$ with non-negative coefficients;
\item $f-g$, where $f,g:\mathbb{R}^m_{\geq 0}\to\mathbb{R}_{\geq 0}$ are polynomials with non-negative coefficients, such that, for some $\varepsilon>0$ and all $x\in\mathbb{R}^m_{\geq 0}$, $f(x)\geq(1+\varepsilon)g(x)$;
\item any function $f$ such that, for every $I$ of size $m'\geq 1$, we have 
$$
f_I(x_1,\ldots,x_{m'})=x_1^{\alpha_1}\ldots x_{m'}^{\alpha_{m'}},
$$
where $\alpha_1,\ldots,\alpha_{m'}\in\mathbb{R}$ are arbitrary constants;
\item $\max\{x_1,\ldots,x_m\}$ and $\min\{x_1,\ldots,x_m\}$.
\end{compactitem}
\label{cor:classes}
\end{corollary}

\paragraph{Limitations of relative Lipschitz.}
Not every polynomial $f:\mathbb{R}^m_{\geq 0}\to\mathbb{R}_{\geq 0}$ is relative Lipschitz --- the assumption in the second bullet point in Corollary~\ref{cor:classes} is essential. The following claim gives illustrative examples of not relative Lipschitz functions.

\begin{claim}
$\,$
\begin{itemize}
\item $2+\sin(x)$ is Lipschitz but not relative Lipschitz on $[0,\infty)$. Moreover, there are Lipschitz functions $f:\mathbb{R}^m_{>0}\to\mathbb{R}_{>0}$ that are not relative Lipschitz such that $f$ is constant on $(\mathbb{R}_{\geq 0}\setminus [0,C))^m$ for some $C$.
\item There are polynomials $f:\mathbb{R}^m_{\geq 0}\to\mathbb{R}_{\geq 0}$ that are not relative Lipschitz. In particular, $(x_1-x_2)^2+x_1^2\in \functionclassrellip$ while $(x_1-x_2)^2\notin \functionclassrellip$ and $(x_1-x_2)^2+x_2\notin \functionclassrellip$.
\end{itemize}
\label{cl:counterexamples}
\end{claim}

\begin{proof}
The fact that $2+\sin(x)$ is  not relative Lipschitz is straightforward --- for example, large enough $x=2\pi n+\pi/2$ and $y=2\pi n-\pi/2$ violate condition~\eqref{eq:asymp_Lip}.

Let $f:\mathbb{R}_{>0}\to\mathbb{R}_{>0}$ be any Lipschitz function such that, for any positive integer $n$,
$$
f(1/n)=n^{-2}\mathbbm{1}_{n\text{ is even}}+2n^{-2}\mathbbm{1}_{n\text{ is odd}}.
$$
Indeed, this function on $\{1/n, n\in\mathbb{Z}_{>0}\}$ has a Lipschitz extension since for positive integers $k_1,k_2$ of different parity,
$$
|f(1/k_1)-f(1/k_2)|=|k_1^{-2}-2k_2^{-2}|=\frac{|2k_2^2-k_1^2|}{k_2^2k_1^2}\stackrel{(*)}\leq 4\frac{|k_2-k_1|}{k_1k_2}=4|k_1^{-1}-k_2^{-1}|.
$$
The inequality (*) is fairly easy to check: if $k_2>k_1$, then 
$$
|2k_2^2-k_1^2|=2k_2^2-k_1^2\leq 2k_2^2\leq 4k_2(k_2-1)\leq 4k_2k_1(k_2-k_1)=4k_1k_2|k_2-k_1|.
$$
If $k_2<k_1$, then
$$
|2k_2^2-k_1^2|\leq k_1^2-2\leq 4k_1(k_1-1)\leq 4k_1k_2(k_1-k_2)=4k_1k_2|k_2-k_1|.
$$

However, $f$ is not relative Lipschitz since
$$
 \frac{|f(1/(2k))-f(1/(2k+1))|}{f(1/(2k))+f(1/(2k+1))}=\frac{|(2k)^{-2}-2(2k+1)^{-2}|}{(2k)^{-2}+2(2k+1)^{-2}}>\frac{1}{3}\left(2\left(\frac{2k}{2k+1}\right)^2-1\right)>1/4,
$$
while
$$
 \frac{|1/(2k)-1/(2k+1)|}{1/(2k)+1/(2k+1)}<\frac{|1/(2k)-1/(2k+1)|}{1/(2k)}=1-\left(\frac{2k}{2k+1}\right)-1=\frac{1}{2k+1}\to 0,\quad k\to\infty,
$$
which contradicts the definition of a relative Lipschitz function.

Let us now switch to polynomials. The fact that $(x_1-x_2)^2+x_1^2\in \functionclassrellip$ follows from the second  bullet point in Corollary~\ref{cor:classes}, since we can take $f=2x_1^2+x_2^2$, $g=2x_1x_2$, and $\varepsilon=0.1$. In order to prove that $f=(x_1-x_2)^2\notin\functionclassrellip$ and $g=(x_1-x_2)^2+x_2\notin\functionclassrellip$, let $k$ be large enough,  $x_1=k+1$, $x_2=k-1$,
$y_1=k-\sqrt{k}$, and $y_2=k+\sqrt{k}$. On the one hand,
$$
 \frac{\|x-y\|}{\|x+y\|}=\frac{\sqrt{k}(1+o_k(1))}{2k(1+o_k(1))}=o_k(1).
$$
However,
$$
 \frac{|f(x)-f(y)|}{f(x)+f(y)}=\frac{4k-4}{4+4k}=\frac{k-1}{k+1}=1-o_k(1),\quad\text{ and }
$$
$$
 \frac{|g(x)-g(y)|}{g(x)+g(y)}= \frac{5k+\sqrt{k}-(k+3)}{k+3+4k+k+\sqrt{k}}=\frac{2}{3}-o_k(1),
$$
contradicting~\eqref{eq:asymp_Lip}.

\end{proof}

All functions in Corollary~\ref{cor:classes} are bounded from above by a polynomial function for large enough $x$ and by an inverse polynomial for $x$ close to 0. 
  Actually, every relative Lipschitz function is uniformly bounded by a power function --- the next assertion is one of the crucial ingridients in the proofs of our main results.

\begin{claim}
\label{cl:poly-bound}
For every relative Lipschitz function $f:\mathbb{R}^m_{>0}\to\mathbb{R}_{>0}$, there exist positive $k$ and $C$ such that for all $x\in\mathbb{R}_{>0}^m$, 
$$
f(x)<C+x_1^k+x_1^{-k}+\ldots+x_m^k+x_m^{-k}.
$$
\end{claim}

\begin{proof}
Let $k,C'\gg C$ be large enough, where $C>0$ is the constant witnessing that $f$ is relative Lipschitz. Let 
$$
g^+=C'+x_1^k+\ldots+x_m^k,\quad
g^-=C'+x_1^{-k}+\ldots+x_m^{-k}.
$$
Assume towards contradiction that there exists $x^*\in\mathbb{R}_{>0}^m$ such that $f(x^*)\geq\max\{g^+(x^*),g^-(x^*)\}$.

Let $S$ be the unit sphere in $\mathbb{R}^m_{>0}$, and let $s=(1/\sqrt{m},\ldots,1/\sqrt{m})\in S$. Let us show that $x^*\notin S$ (and actually far away from $S$). Indeed, assume the opposite, and let $y\neq x\in S$. Divide the geodesic in $\mathbb{R}^m$ between $x^*$ and $y$ into $4C+1$ parts via $y_1,\ldots,y_{4C}$. Then the distance between any two consecutive points is at most $1/{2C}$, implying $|f(y_i)-f(y_{i-1})|\leq\frac{1}{2}(f(y_i)+f(y_{i-1}))$. Therefore, $f(y_i)\leq 3f(y_{i-1})$, implying the universal upper bound $f(x^*)\leq 3^{4C+1}f(y)$.

Assume first $\|x^*\|>1$. In this case we may also assume that $f(x^*)\geq g^+(x^*)$. Let $\alpha\in S$ be such that $x^*=\alpha\cdot t^*$, where $t^*>1$. Let $\mathcal{T}\subset(1,\infty)$ be the set of all $t$ such that $f(\alpha\cdot t)\geq g^+(\alpha\cdot t)$. The set $\mathcal{T}$ is closed since $f-g^+$ is continuous due to Claim~\ref{cl:continuous}. Therefore, we may assume that there exists $x=\alpha\cdot t\in\alpha\cdot\mathcal{T}$ such that $f(x)=g^+(x)$ and $f(y)<g^+(y)<g^+(x)$ for $y=x(1-1/k)$ (recall that we can choose $k$ as large as we need). Then,
$$
 f(x)-f(y)\geq g^+(x)-g^+(y)=\sum_{i=1}^m\left((\alpha_i t)^k-(\alpha_i t(1-1/k))^k\right)\geq (1-1/e) t^k\sum_{i=1}^m\alpha_i^k;
$$
$$
 f(x)=g^+(x)=\sum_{i=1}^m(\alpha_i t)^k=t^k\sum_{i=1}^m\alpha_i^k; \quad\text{ and }\quad
 \frac{\|x-y\|}{\|x+y\|}=\frac{\|x/k\|}{\|x(2-1/k)\|}=\frac{1}{2k-1}.
$$
We get
$$
 \frac{|f(x)-f(y)|}{f(x)+f(y)}\geq\frac{f(x)-f(y)}{2f(x)}\geq \frac{1-1/e}{2}=(1-1/e)\cdot\frac{2k-1}{2}\cdot\frac{\|x-y\|}{\|x+y\|}
$$
--- a contradiction since $k$ can be chosen arbitrarily large. In particular, $k>3C$.

Now, let $\|x^*\|<1$ and $f(x^*)\geq g^-(x^*)$. As above, let $\alpha\in S$ be such that $x^*=\alpha\cdot t^*$, where $t^*<1$. We may find $x=\alpha\cdot t$ such that $f(x)=g^-(x)$ and $f(y)<g^-(y)<g^-(x)$, where $y=x(1+1/k)$. Then,
$$
 f(x)-f(y)\geq g^-(x)-g^-(y)=\sum_{i=1}^m\left((\alpha_i t)^{-k}-(\alpha_i t(1+1/k))^{-k}\right)\geq (1-2/e) t^{-k}\sum_{i=1}^m\alpha_i^{-k};
$$
$$
 f(x)=g^-(x)=\sum_{i=1}^m(\alpha_i t)^{-k}=t^{-k}\sum_{i=1}^m\alpha_i^{-k}; \quad\text{ and }\quad
 \frac{\|x-y\|}{\|x+y\|}=\frac{\|x/k\|}{\|x(2+1/k)\|}=\frac{1}{2k+1}.
$$
As above,
$$
 \frac{|f(x)-f(y)|}{f(x)+f(y)}\geq \frac{f(x)-f(y)}{2f(x)}\geq\frac{1-2/e}{2}=(1-2/e)\cdot\frac{2k+1}{2}\cdot\frac{\|x-y\|}{\|x+y\|}
$$
--- a contradiction again, completing the proof.
\end{proof}

Even though asymptotic behaviour of a relative Lipschitz functions as $x\to 0$ or $x\to\infty$, is bounded by a polynomial, it may differ from the behaviour of any power function. For instance, the function $f(x)=\ln(1+x)$ is relative Lipschitz on $(0,\infty)$. Moreover, the speed of decay of a relative Lipschitz function may be significantly higher when the limit of $f(x)$ differs from $0$ and $\infty$. For example, the sigmoid function $f(x)=(1+e^{-x})^{-1}$ is relative Lipschitz on $[0,\infty).$ More generally, it is easy to see that every differentiable $f:\mathbb{R}_{>0}\to\mathbb{R}_{>0}$ such that $xf'(x)<Cf(x)$ for some $C>0$ and all $x>0$ is relative Lipschitz (see Appendix~\ref{app:rellip}).

\paragraph{Asymptotically polynomial functions.} There exists a term $\tau\in\Agglang_{\functionclassrellip}$ such that $\mathbb{E}\tau$ is bounded but does not converge as $n\to\infty$. This follows from the existence of a bounded $f:\mathbb{R}_{\geq 0}\to\mathbb{R}_{\geq 0}\in\functionclassrellip$ such that $\lim_{x\to\infty}f(x)$ does not exist. For example, $f(x)=2+\sin(\ln(2+x))$ is such a function --- see Appendix~\ref{app:no-limit}. In order to restrict the family of connectives to ``asymptotically consistent'' functions, we introduce the following definition.

\begin{definition}[Asymptotically Polynomial] \label{def:asymptpoly}
A function $f:\mathbb{R}^m_{>0}\to\mathbb{R}_{>0}$ is {\it asymptotically polynomial}, if, for any positive constants $c_1,\ldots,c_m$, real constants $\gamma_1,\ldots,\gamma_m$, and sequence 
$$
a_k=((c_1+o(1))k^{\gamma_1},\ldots,(c_m+o(1))k^{\gamma_m}),
$$
there exist $c>0$ and $\gamma\in\mathbb{R}$ such that 
$$
f(a_k)=(c+o(1))k^{\gamma}.
$$

A function $f:\mathbb{R}^m_{\geq 0}\to\mathbb{R}_{\geq 0}$ is {\it asymptotically polynomial}, if every  specialisation is either positive-valued and asymptotically polynomial, or identically equal to zero.

Let $\functionclasspoly$ be the set of asymptotically polynomial relative Lipschitz functions $f:\mathbb{R}^m_{\geq 0}\to\mathbb{R}_{\geq 0}$ over all integers $m\geq 1$.
\end{definition}

We notice that all functions from Corollary~\ref{cor:classes} are asymptotically polynomial. The sigmoid function is asymptotically polynomial and relative Lipschitz as well. 
Actually, none of the classes $\functionclasspoly$, $\functionclassrellip$, $\Lipsch$ is contained in any of the others. From Claim~\ref{cl:counterexamples}, we know that this is the case for the classes $\functionclassrellip$ and $\Lipsch$. The remaining non-inclusions are evident from the following examples:

\begin{itemize}

\item An asymptotically polynomial function is not necessarily continuous --- therefore, not every asymptotically polynomial function is Lipschitz. 

\item On the other hand, all polynomial functions are asymptotically polynomial but not every polynomial function $f:\mathbb{R}^m_{\geq 0}\to\mathbb{R}_{\geq 0}$ is relative Lipschitz, as follows from Claim~\ref{cl:counterexamples}. 

\item $\ln(1+x)$ is relative Lipschitz but not asymptotically polynomial.

\end{itemize}

\section{Dense \erdosrenyi \ graphs}
\label{sec:dense-all}

Let $p=\mathrm{const}\in(0,1)$ and let $\mathbf{G}_n\sim G(n,p)$.
Recall the following concentration result for the language with average-based aggregates from \cite{uslics}. 

\begin{fact}[\cite{uslics}] For every closed term $\tau$ from the language  $\Aggo{\Lipsch, \Mean, \LMean, \supl}$, its evaluation $[\![\tau]\!]_{\mathbf{G}_n}$
converges in probability to a constant.
\end{fact}

\subsection{Main results}

Our first main result asserts that every closed term $\tau$ in $\Agglang_{\functionclassrellip}$ has $\tau=[\![\tau]\!]_{\mathbf{G}_n}$ concentrated around its mean. 

\begin{theorem} \label{th:dense-concentration}
For every closed term $\tau$ in $\Agglang_{\functionclassrellip}$ either
\begin{compactitem}
\item  $\whp$ $\tau=0$, and $\tau\stackrel{L_p}\to 0$ for every $p>0$ ($\tau$ is typically zero), or
\item $\mathbb{E}\tau>0$, $\whp$ $\tau>0$, and, for every $p>0$,
 $\tau/\mathbb{E} \tau\stackrel{L_p}\to 1$ as $n\to\infty$.
\end{compactitem}

\end{theorem}

Our second main result states that, if we further restrict the set of connective functions to $\functionclasspoly$, then $\mathbb{E}\tau=(c+o(1))n^{\gamma}$.

\begin{theorem} \label{th:dense-limits}
For every closed term $\tau$ in $\Agglang_{\functionclasspoly}$ either
\begin{compactitem}
\item  $\whp$ $\tau=0$, and $\tau\stackrel{L_p}\to 0$ for every $p>0$, or
\item there exist $c>0$ and $\gamma\in\mathbb{R}$ such that $\mathbb{E}\tau=(c+o(1))n^{\gamma}$. 
\end{compactitem}

\end{theorem}

In particular, if $\sup_n\mathbb{E}\tau<\infty$, then there exists $\lim_{n\to\infty}\mathbb{E}\tau=:c$ and $\tau\stackrel{P}\to c$ as $n\to\infty$. More precisely, if $\gamma=0$, we get that $\tau\stackrel{P}\to c$, and if $\gamma>0$, then $\tau$ diverges but concentrates.

\begin{remark}
\label{rk:no-inf}
Analogously to the treatment of the existential quantifier in first-order logic, $\infl$ can be removed from the language without any loss of expressive power. Indeed, let $f\in\functionclasspoly$ be a unary function defined as $f(0)=1$ and $f(x)=0$ for all $x>0$. Let $g\in\functionclasspoly$ be the inverse function: $g(x)=1/x$ if $x>0$ and $g(0)=0$. Finally, let $P(x_1,x_2)=x_1x_2\in\functionclasspoly$. Then 
$\infl_u\tau(u)=P(f(\supl_u f(\tau(u))), g(\supl_u g(\tau(u))))$.
\end{remark}

\begin{remark}
All of the summation-based languages in the paper, in particular $\Agglang_{\functionclasspoly}$, can express
the term $\sum_v 1$ ---  the total number of vertices $n$ --- ignoring the graph structure. We can then compute the
output of connective functions applied
to $n$. 
\end{remark}

\begin{remark}
We restricted the connective functions to output non-negative values and to be relative Lipschitz. We now argue that these restrictions are necessary for the validity of Theorem~\ref{th:dense-concentration}. 
 Let us support this statement by the following two negative examples.

First, let $\tau_1$ be the difference between the number of edges and $\frac{n^2}{2}$, which can be expressed in $\Agglang_{\functionclassrellip\cup\{f_1(x_1,x_2)=x_1-x_2\}}$ since the function $g(x)=\frac{x^2}{2}$ is relative Lipschitz. Then $\mathbb{E}|\tau_1|=\Theta(n)$ and $\tau_1/n$ coverges to a normal random variable in distribution, thus $\tau_1$ does not concentrate.

It is easy to make the previous example non-negative by defining $f_2(x_1,x_2)=(x_1-x_2)^2$ and $\tau_2=\tau_1^2\in\Agglang_{\functionclassrellip\cup\{f_2\}}$. We still get that $\mathbb{E}\tau_2=\Theta(n^2)$ and $\tau_2/n^2$ has a non-degenerate asymptotic distribution. It follows that the function $f_2$ does not belong to $\functionclassrellip$, which is also established in Claim~\ref{cl:counterexamples}.\footnote{It may seem that to guarantee $f\in \functionclassrellip$ for a polynomial $f$, it would be sufficient to rule out non-trivial roots (and it is actually the case for univariate polynomials --- every polynomial $f:\mathbb{R}_{>0}\to\mathbb{R}_{>0}$ belongs to $\functionclassrellip$). However, Claim~\ref{cl:counterexamples} gives another counterexample $f_3(x_1,x_2)=(x_1-x_2)^2+x_2$ that has a unique root $(0,0)$. In particular, the term $(|E(G)|-n^2/2)^2+n^2/2$ belongs to $\Agglang_{\functionclassrellip\cup\{f_3\}}$ and does not satisty the conclusion of Theorem~\ref{th:dense-concentration}.}
\end{remark}

\paragraph{Applications.}
Our results imply concentration of subgraph counts and extension counts studied in, e.g., \cite[Chapter 3]{Janson}, \cite[Chapter 4]{Bol-book}, \cite{Spencer_count},~\cite{Warnke}. In particular, for a fixed graph $H$, let $X_H$ be the number of subgraphs in $\mathbf{G}_n$ isomorphic to $H$. For a MRG $H[R]$ with a set of roots $R=\{v_1,\ldots,v_r\}$, with no edges between the roots, and a set $U=\{u_1,\ldots,u_r\}\subset[n]$, let $X_{R,H}(U)$ be the set of subgraphs $H'\subset\mathbf{G}_n$ such that $U\subset V(H')$ and there exists a function $f:V(H)\to V(H')$ that maps every $v_i$ to $u_i$ and preserves all edges in $E(H)$. In other words, $X_{R,H}(U)$ is the number of MRGs $H'[U]$ isomorphic to $H[R]$ in $\mathbf{G}_n$, and it is easy to represent this count by a term in our language. In particular, when $H$ consists of a single edge and a single root, $X_{R,H}(u)$ is the degree of $u$ in $\mathbf{G}_n$. 
\begin{corollary}
\label{cor:dense-cor}
For every $H$,
\begin{compactitem}
    \item $X_H$ is concentrated around 
    $$
    \mathbb{E}X_{H}=\frac{1}{\mathrm{aut}(H)}(n)_{|V(H)|}\cdot p^{|E(H)|},
    $$
    where $\mathrm{aut}(H)$ is the number of automorphisms of $H$, and
    \item for every set of roots $R\subset V(H)$, $\max_U X_{R,H}(U)$ is concentrated around 
    $$
    \mathbb{E}\max_U X_{R,H}(U)=\left(\frac{1}{\mathrm{aut}(H[R])}+o(1)\right)n^{|V(H)|-r}\cdot p^{|E(H)|},
    $$ 
    where $\mathrm{aut}(H[R])$ is the number of automorphisms of MRG $H[R]$.
\end{compactitem}
\end{corollary}
We stress that, for $X_H$ and for $\max_U X_{R,H}(U)$ (although, for a restricted class of MRGs $H[R]$) tighter concentration results are known~\cite{Bollobas_degrees,Zhuk-ext1,Zhuk-ext2,smallsubgraphs,SS,Zhuk-ext3}. Nevertheless, even though we do not specify the rate of convergence in Theorem~\ref{th:dense-limits}, our proof technique allows to establish optimal bounds, up to a constant factor. In particular, it is possible to extract that, for every 
$\tau$ in $\Agglang_{\functionclasspoly}$, there exists $C>0$ such that $|\tau-\mathbb{E}\tau|\leq C\mathbb{E}\tau\sqrt{\ln n/n}$ $\whp$, which gives the right order of magnitude of the second-order term, for example, for the maximum degree.

\subsection{Proofs of Theorem~\ref{th:dense-concentration} and Theorem~\ref{th:dense-limits}}
\label{sc:dense-proofs}

We will abstract concrete tuples in graphs by their atomic types. We will then establish  a bound on the number of extensions of a given atomic type within
random graphs.  This extension bound is used in the main inductive construction of this section --- Lemma~\ref{lem:dense_main}, which is a quantitative version of Theorem~\ref{th:dense-concentration} and Theorem~\ref{th:dense-limits}, and will easily be seen to imply both of them.

\begin{definition}[Atomic Types] \label{def:atomictype}
An {\it atomic type} is a labeled graph on a vertex set $[k]$, where $k$ is a positive integer. Let $k$ be a positive integer, let $G$ be a graph on $[n]$, let $\overline{u}=(u_1,\ldots,u_k)\in([n])_k$, and let $H$ be a graph on $[k]$. We say that $\overline{u}$ has atomic type $H$ in $G$ if $H[(1, \ldots,k)]$ is isomorphic to  $(G\upharpoonright\overline{u})[\overline{u}]$\footnote{In other words, the bijection $f:[k]\to\overline{u}$ that maps every $i\in[k]$ to $u_i$ is an isomorphism of $H$ and $G\upharpoonright\overline{u}$.}. 
Let $\atomictypes$ be the set of all atomic types. For a graph $G$ on $[n]$, let $\atomictypeof^G:\cup_k([n])_k\to\atomictypes$ map tuples of vertices in $G$ to their atomic types.  In what follows, when we omit the superscript, we mean $G=\mathbf{G}_n$, and the corresponding object is now random:  we write $\atomictypeof=\atomictypeof^{\mathbf{G}_n}$. 
We let $\atomictypes_k$  be the set of all atomic types of $k$-tuples, that is for any $k$-tuple $\overline{u}\in([n])_k$,  we have $\atomictypeof(\overline{u})\in\mathcal{A}_k$. We also let $\varnothing\in\mathcal{A}$ be an atomic type of graph on an empty set of vertices.  

\end{definition}

\begin{lemma}[Atomic type extension percentages stabilize]\label{lem:extensionsstabilizedense}
Let $k$ be a non-negative integer and let $k'$ be a positive integer. Let $\mathbf{t}_0\in\atomictypes_k$ be an atomic type. Let $\atomictypes_{k'}[t_0]\subset\atomictypes_{k+k'}$ be the set of atomic types $\mathbf{t}$  such that $\mathbf{t}\upharpoonright[k_0]=\mathbf{t}_0$.

There exists a map $\atppercent:\atomictypes_{k'}[\mathbf{t}_0]\to\mathbb{R}_{\geq 0}$ 
such that the following holds for every $k$-tuple $\overline{u}\in ([n])_k$, possible empty if $k=0$, with probability $1-n^{-\omega(1)}$. If $\atomictypeof(\overline{u})=\mathbf{t}_0$, then, for every $\mathbf{t}\in\atomictypes_{k'}[\mathbf{t}_0]$, the number of tuples $\overline{v}\in ([n]\setminus\overline{u})_{k'}$ such that $\atomictypeof(\overline{u},\overline{v})=\mathbf{t}$ is at most $\left(\atppercent(\mathbf{t})\cdot n^{k'-1/2}(\ln n)^{0.9}\right)$-far from $\atppercent(\mathbf{t})\cdot n^{k'}$. If $\atppercent(\mathbf{t})=0$, then the number of such tuples is 0 a.s.

\end{lemma}

\begin{proof}
 Fix $\mathbf{t}\in\atomictypes_{k'}[\mathbf{t}_0]$ and $\overline{u}\in ([n])_{k}$. For $\overline{v}\in ([n]\setminus \overline{u} )_{k'}$, let 
$$
 \atppercent(\mathbf{t}):=\mathbb{P}(\atomictypeof(\overline{u},\overline{v})=\mathbf{t}\mid\atomictypeof(\overline{u})=\mathbf{t}_0).
$$
Note that $\atppercent(\mathbf{t})$ does not depend on the choice of $\overline{v}$.

If $\atppercent(\mathbf{t})=0$, then the number tuples $\overline{v}$  such that $(\overline{u},\overline{v})$ has type $\mathbf{t}$ equals 0 a.s.

Assume $\atppercent(\mathbf{t})>0$. Let $Y$ be the number of $k'$-tuples $\overline{v}$  in $([n]\setminus\overline{u})_{k'}$ such that $\atomictypeof(\overline{u},\overline{v})=\mathbf{t}$, subject to $\{\atomictypeof(\overline{u})=\mathbf{t}_0\}$.

Recall that for a graph characteristic $\chi:2^{{[n]\choose 2}}\to\mathbb{R}$, the edge martingale of $\chi(\mathbf{G}_n)$ is the Doob martingale\footnote{Given a random variable $X$, defined on a measurable space $(\Omega,\mathcal{F})$, with $\mathbb{E}|X|<\infty$ and a suquence of $\sigma$-algebras $\mathcal{F}_0\subset\mathcal{F}_1\subset\ldots\subset\mathcal{F}$ (which is called a filtration), the discrete-time random process $(X_i=\mathbb{E}(X\mid \mathcal{F}_i))_{i\in\mathbb{Z}_{\geq 0}}$ is called a Doob martingale of $X$ with respect to the filtration $\mathcal{F}_0\subset\mathcal{F}_1\subset\ldots$.} of $\chi(\mathbf{G}_n)$ with respect to the filtration induced by the random graph process\footnote{A discrete-time random process $X_0,X_1,\ldots$ induces the sequence of $\sigma$-algebras $\sigma(X_0),\sigma(X_0,X_1),\ldots$, which is called a filtration. Recall that, for a random variable $\xi:\Omega\to E$, where $(\Omega,\mathcal{F})$ and $(E,\mathcal{E})$ are measurable spaces, $\sigma(\xi)=\{\xi^{-1}(A)\mid A\in\mathcal{E}\}$.} $(\mathbf{G}(i,j))_{1\leq i<j\leq n}$, where $\mathbf{G}(i,j)$ contains all edges $\{u,v\}$ of $\mathbf{G}_n$ where $u\leq i$ and $v\leq j$. We now apply the bounded-difference inequality~\cite[Corollary 2.27]{Janson} to the edge martingale of $Y$. First of all, 
$$
 \mathbb{E}Y=\atppercent(\mathbf{t})\cdot(n-k)_{k'}.
$$ 
Moreover, changing a single adjacency relation between two vertices outside of $\overline{u}$, changes types of $O(n^{k'-2})$ tuples (those containing both vertices in the relation) and
changing a single adjacency relation between a vertex in $\overline{u}$ and a vertex outside of $\overline{u}$, changes types of $O(n^{k'-1})$ tuples. Therefore,
\begin{align*}
 \mathbb{P}\biggl(|Y-\mathbb{E}Y|>n^{k'-1/2}\cdot(\ln n)^{0.9}/2\mid\atomictypeof(\overline{u})=\mathbf{t}_0\biggr)
 &\leq 2\exp\left(-\frac{n^{2k'-1}(\ln n)^{1.8}}{8(kn\cdot (n^{k'-1})^2+n^2
 \cdot (n^{k'-2})^2)}\right)
 \\
 &=\exp(-\Omega((\ln n)^{1.8}))=n^{-\omega(1)}.
\end{align*}
We get
\begin{align*}
 \mathbb{P}\left(|Y-\atppercent(\mathbf{t})\cdot n^{k'}|>n^{k'-1/2}(\ln n)^{0.9}\mid\atomictypeof(\overline{u})=\mathbf{t}_0\right)=n^{-\omega(1)},
\end{align*}
as required.
\end{proof}

We naturally extend the notion of types and the set $\atomictypes_k$ from tuples of distinct elements to tuples from $[n]^k$. 

We are now ready to state our main lemma which is a quantitative version of Theorems~\ref{th:dense-concentration}~and~\ref{th:dense-limits} for {\it open terms} in the language $\Agglang_{\functionclassrellip}$.

\begin{lemma}[Inductive Term Approximation] \label{lem:dense_main}
For every term $\tau \in \Agglang_{\functionclassrellip}$ with $k$ free variables, there exist constant $K>0$ and maps $\varphi_n:\atomictypes_k\to \mathbb{R}_{\geq 0}$ for every $n \in \nats$,  such that, for every $\overline{u}\in[n]^k$ and every $\mathbf{t}\in\atomictypes_k$, we have 
$$
\mathbb{P}(n^{-K}\leq\tau(\overline{u})\leq n^{K}\mid\tau(\overline{u})\neq 0)=1
$$ 
and one of the following  holds:
\begin{compactitem}
\item $\varphi_n(\mathbf{t})\equiv\mathrm{0}$ and $\tau$ is typically zero for tuples of type $\mathbf{t}$:  
 $$
 \mathbb{P}(\tau(\overline{u})=\varphi_n(\mathbf{t})=0\mid\atomictypeof(\overline{u})=\mathbf{t})
 =1-n^{-\omega(1)},
 $$ 
 \item $\varphi_n(\mathbf{t})>n^{-K}$, and $\tau$ is concentrated around $\varphi_n(\mathbf{t})$ for tuples of type $\mathbf{t}$:
\begin{equation}
\mathbb{P}\left(|\tau(\overline{u})-\varphi_n(\mathbf{t})|>\frac{\ln n}{\sqrt{n}}\cdot \varphi_n(\mathbf{t})\mid \atomictypeof(\overline{u}) =\mathbf{t}\right)=n^{-\omega(1)}.
\label{eq:concentratio-dense-main}
\end{equation}
\end{compactitem}
Moreover, if $\tau \in \Agglang_{\functionclasspoly}$, and $\varphi_n(\mathbf{t})>0$, then, for some positive $c=c(\mathbf{t})$ and real $\gamma=\gamma(\mathbf{t})$,  $$
\varphi_n(\mathbf{t})=(c+o(1))n^{\gamma}.
$$
\end{lemma}

We first derive Theorem~\ref{th:dense-concentration} and Theorem~\ref{th:dense-limits} from Lemma~\ref{lem:dense_main}.

\begin{proof}[Proofs of Theorem~\ref{th:dense-concentration} and Theorem~\ref{th:dense-limits}]
Fix $\tau$. Due to Lemma ~\ref{lem:dense_main}, applied to the empty type, there exist a sequence $\varphi_n\in\mathbb{R}_{\geq 0}$ and constant $K>0$ such that 
\begin{compactitem}
\item either $\varphi_n\equiv 0$, $\mathbb{P}(\tau=0)=1-n^{-\omega(1)}$, and $\tau\leq n^K$ a.s. 

\item or $n^{-K}\leq \tau\leq n^K$ a.s., $\varphi_n>n^{-K}$, and, with probability at least $1-n^{-\omega(1)}$,
\begin{equation}
 |\tau-\varphi_n|\leq \frac{\ln n}{\sqrt{n}}\cdot \varphi_n.
\label{eq:tau_concentration_phi}
\end{equation}
\end{compactitem}
If $\varphi_n\equiv 0$, then 
$$
\mathbb{E}\tau^p=\mathbb{E}\left(\tau^p\cdot\mathbbm{1}_{0<\tau\leq n^K}\right)\leq n^{pK}\cdot\mathbb{P}(\tau\neq 0)=n^{-\omega(1)}.
$$ 
Thus we have achieved the first bullet item of the conclusions of Theorem~\ref{th:dense-concentration} and Theorem~\ref{th:dense-limits}. 

Now, let $\varphi_n>n^{-K}$. Then 
\begin{align}
\varphi_n\left(1-\frac{\ln n}{\sqrt{n}}\right)\left(1-n^{-\omega(1)}\right)&\leq\mathbb{E}\tau=\mathbb{E}\tau\cdot\mathbbm{1}_{0\leq\tau\leq n^{K}}\notag\\
&\leq\mathbb{E}\tau\cdot\mathbbm{1}_{|\tau/\varphi_n-1|\leq \ln n/\sqrt{n}}+n^{-\omega(1)}\notag\\
&<\varphi_n\left(1+\frac{\ln n}{\sqrt{n}}\right)+n^{-\omega(1)}.
\label{eq:tau-exp-concentration}
\end{align}
 By combining this with~\eqref{eq:tau_concentration_phi}, we get that with probability $1-n^{-\omega(1)}$,
\begin{align*}
 1-\frac{(\ln n)^2}{\sqrt{n}}<\frac{1-\ln n/\sqrt{n}}{1+\ln n/\sqrt{n}+n^{-\omega(1)}}\leq\frac{\tau}{\mathbb{E}\tau}
 \leq\frac{1+\ln n/\sqrt{n}}{1-\ln n/\sqrt{n}-n^{-\omega(1)}}<1+\frac{(\ln n)^2}{\sqrt{n}}
\end{align*}
and $\tau/\mathbb{E}\tau=n^{O(1)}$ a.s. We conclude that
$$
 \mathbb{E}\left|\frac{\tau}{{\mathbb E}\tau}-1\right|^p\leq \left(\frac{(\ln n)^2}{\sqrt{n}}\right)^p+n^{O(1)-\omega(1)}=o(1).
$$
Thus we have achieved the second bullet item of Theorem~\ref{th:dense-concentration}.

It only remains to observe that the second bullet item of Theorem~\ref{th:dense-limits} follows immediately from
the last conclusion of Lemma~\ref{lem:dense_main}: since $\varphi_n(t)=(c+o(1))n^{\gamma}$, the same holds true for $\mathbb{E}\tau$ due to~\eqref{eq:tau-exp-concentration}. This completes the proof.

\end{proof}

\begin{proof}[Proof of Lemma ~\ref{lem:dense_main}.]

We prove Lemma~\ref{lem:dense_main} by induction on the grammar for terms. 
Moreover, we prove a slightly stronger assertion --- that the factor $\frac{\ln n}{\sqrt{n}}$ in the bound~\eqref{eq:concentratio-dense-main} can be replaced with $\lambda\cdot\frac{\ln n}{\sqrt{n}}$ for any $\lambda>0$. 
We give the proof for atomic cases, and for the cases of function, summation, and $\supl$ applications. 
 We do not need to consider the case $\tau(\overline{u})=\infl_v\tau'(\overline{u},v)$ since it can be expressed as a combination of connective functions from $\functionclasspoly$ and $\supl$ --- see Remark~\ref{rk:no-inf}.

Fix $\lambda>0$.

\paragraph{Atomic cases.}

\begin{itemize}

\item $\tau\equiv\mathrm{const}$. In this case, the assertion is obvious for $\varphi_n=\tau$.

\item $\tau(u_1,u_2)=\mathbbm{1}_{ \mathrm{E}(u_1,u_2)}$ is a term with two free variables. We let $\varphi_n(\mathbf{t})=1$ if $\mathbf{t}$ has an edge and $\varphi_n(\mathbf{t})=0$ if $\mathbf{t}$ does not have an edge. Then $\tau(u_1,u_2)=\varphi_n(\atomictypeof(u_1,u_2))$ for every pair of vertices $(u_1,u_2)\in[n]^2$. The assertion of Lemma~\ref{lem:dense_main} clearly follows. 

\item The case of $\tau(u_1,u_2)=\mathbbm{1}_{u_1=u_2}$ is analogous to the previous edge case.

\end{itemize}

\paragraph{Function application.} $\tau=f(\tau_1,\ldots,\tau_m)$, where $f\in\functionclassrellip$. Let $\tau$ and $\tau_i$ have $k$ and $k_i$ free variables respectively. Substituting  $\overline{u}\in([n])_k$ as an input of $\tau$, we derive arguments of $\tau_1,\ldots,\tau_m$ --- we denote the respective subtuples of $\overline{u}$ by $\overline{u}_i$. Fix an atomic type $\mathbf{t}\in\atomictypes_k$. The event $\atomictypeof(\overline{u})=\mathbf{t}$ defines the types $t_i$ of $\overline{u}_i$, and we assume that this event holds in what follows. 

By induction, we get that sequences $\varphi^i_n=\varphi^i_n(\mathbf{t}_i)$, $i\in[m]$, have been already defined and that terms $\tau_1,\ldots,\tau_m$ satisfy the assertion of Lemma~\ref{lem:dense_main} with the error term $\ln n/\sqrt{n}$ replaced by $\lambda\ln n/\sqrt{n}$.

In particular, there are constants $K_1,\ldots,K_m$ such that a.s. $\tau_i\leq n^{K_i}$ for every $i\in[m]$. We apply Theorem~\ref{cl:function-concentration} with 
$$
\xi_i:=\tau_i, \, \varphi_i:=\varphi^{(i)}_n, \, \delta:=n^{-\omega(1)}, \, \delta':=\lambda\cdot\frac{\ln n}{\sqrt{n}},
$$
and 
$$
K:=\max\{K_1,\ldots,K_m\}.
$$ 
We immediately get that, for some constants $a,b>0$ (that only depend on $f$), 
$$
\mathbb{P}(\tau\in[K^{-b},K^b]\mid\tau\neq 0)=1,
$$
and either $\varphi_n\equiv 0$ and $\mathbb{P}(\tau=0)=1-n^{-\omega(1)}$, or $\varphi_n\geq K^{-b}$ and 
$$
\mathbb{P}\left(|\tau-\varphi_n|>(a\lambda)\frac{\ln n}{\sqrt{n}}\varphi_n\right)=1-n^{-\omega(1)}.
$$

Finally, we assume that $\tau \in \Agglang_{\functionclasspoly}$ and, in particular, $f\in\functionclasspoly$. Let  $I=\{i_1,\ldots,i_{m'}\}\subset[m]$ be the set of  $i\in[m]$ such that $\varphi^i_n\not\equiv 0$. We also assume, by induction, that, for every $i\in I$, $\varphi_n^i=(c_i+o(1))n^{\gamma_i}$ for some positive $c_i$ and real $\gamma_i$. We  must prove there is $c>0$ and $\gamma\in\mathbb{R}$ such that $\varphi_n=(c+o(1))n^{\gamma}$. Since $f\in\functionclasspoly$, we get that $f_I$ is asymptotically polynomial. Then, either $f_I\equiv 0$ and $\varphi_n=0$, or 
\begin{align*}
\varphi_n=f(\varphi^1_n,\ldots,\varphi^m_n)=f_I(\varphi^I_n)
=f_I\biggl((c_{i_1}+o(1))n^{\gamma_{i_1}},\ldots,(c_{i_{m'}}+
o(1))n^{\gamma_{i_{m'}}}\biggr)
=(c+o(1))n^{\gamma},
\end{align*}
for some $c>0$ and $\gamma\in\mathbb{R}$, by the definition of asymptotically polynomial functions. This completes the proof.

\paragraph{Summation approximation.} $\tau(\overline{u})=\sum_v\tau'(v,\overline{u})$, where, by the induction hypothesis, there exist constant $K'>0$ and a map $\varphi'_n:\atomictypes_{k+1}\to\mathbb{R}_{\geq 0}$ such that, for every type $\mathbf{t}$, every tuple $(\overline{u},v)$, and every $\lambda>0$,
\begin{align*}
\mathbb{P}\left(|\tau'(\overline{u},v)-\varphi'_n(\mathbf{t})|>\lambda\frac{\ln n}{\sqrt{n}}\varphi'_n(\mathbf{t})\mid \atomictypeof(\overline{u},v)=\mathbf{t}\right)
=n^{-\omega(1)}
\end{align*}
and $\mathbb{P}(n^{K'}\leq\tau'(\overline{u},v)\leq n^{K'}\mid \tau'(\overline{u},v)\neq 0)=1$. Moreover, either $\varphi'_n(\mathbf{t})\equiv 0$ or $\varphi'_n(\mathbf{t})\geq n^{-K'}$.

Fix $\overline{u}\in([n])_k$ and type $\mathbf{t}_0\in\atomictypes_k$. In what follows, we assume that the event $\atomictypeof(\overline{u})=\mathbf{t}_0$ holds and define the value of $\varphi_n(\mathbf{t}_0)$. Let $\atomictypes[\mathbf{t}_0]\subset\atomictypes_{k+1}$ be the set of atomic types achievable by $(\overline{u},v)$, subject to $\atomictypeof(\overline{u})=\mathbf{t}_0$. 
 Due to Lemma~\ref{lem:extensionsstabilizedense}, there exists a map $\atppercent:\atomictypes[\mathbf{t}_0]\to\mathbb{R}_{\geq 0}$ such that, for every $\mathbf{t}\in\atomictypes[\mathbf{t}_0]$, the number of vertices $v\notin\overline{u}$ such that $\atomictypeof(v,\overline{u})=\mathbf{t}$ equals $\atppercent(\mathbf{t})\cdot n(1\pm o(n^{-1/2}\ln n))$ with probability $1-n^{-\omega(1)}$. 
 We also
 assume $\atppercent(\mathbf{t})\neq 0$ for all $\mathbf{t}\in\mathcal{A}[\mathbf{t}_0]$ since types $\mathbf{t}$ with $\atppercent(\mathbf{t})=0$ do not contribute to $\tau$ with probability $1-n^{-\omega(1)}$.

Let $\atomictypes_{\mathrm{in}}[\mathbf{t}_0]\subset\atomictypes[\mathbf{t}_0]$ be the set of atomic types of tuples $(\overline{u},v)$ such that $v\in\overline{u}$ 
 and let $\atomictypes_{\mathrm{out}}[\mathbf{t}_0]=\atomictypes[\mathbf{t}_0]\setminus\atomictypes_{\mathrm{in}}[\mathbf{t}_0]$. For every $\mathbf{t}\in\atomictypes[\mathbf{t}_0]$, such that $\varphi'_n(\mathbf{t})\equiv 0$, we have
$$
\mathbb{P}(\tau'(\overline{u},v)=0\mid\atomictypeof(\overline{u},v)=\mathbf{t})=1-n^{-\omega(1)}.
$$
Therefore,
by the union bound, if, for every $\mathbf{t}\in\atomictypes[\mathbf{t}_0]$, we have that $\varphi'_n(\mathbf{t})\equiv 0$, then $\tau(\overline{u})=0$ with probability $1-n^{-\omega(1)}$. 

Therefore, we can assume that there exists $\mathbf{t}\in\atomictypes[\mathbf{t}_0]$ such that $\varphi'_n(\mathbf{t})>0$. 
 If all such $\mathbf{t}$ are from $\atomictypes_{\mathrm{in}}[\mathbf{t}_0]$, then, by the union bound, with probability $1-n^{-\omega(1)}$, $\tau(\overline{u})=\sum_{v\in\overline{u}}\tau'(v,\overline{u})$, which reduces to the case of function application, since linear functions with positive coefficients belong to $\functionclasspoly$.

Thus, we can assume that there exists $\mathbf{t}\in\atomictypes_{\mathrm{out}}[\mathbf{t}_0]$ such that $\varphi'_n(\mathbf{t})\geq n^{-K'}>0$.  
We are now ready to define
\begin{align}
\label{eq:phi-sum}
\varphi_n(\mathbf{t}_0)&:=n\sum_{\mathbf{t}\in\atomictypes_{\mathrm{out}}[\mathbf{t}_0]}\varphi'_n(\mathbf{t})\atppercent(\mathbf{t})+\sum_{v\in\overline{u}}\varphi'_n(\atomictypeof(\overline{u},v)).
\end{align}
By the union bound, with probability $1-n^{-\omega(1)}$,
\begin{align*}
 |\tau(\overline{u})-\varphi_n(\mathbf{t}_0)|&=
 \left|\sum_{v\in\overline{u}}\tau'(v,\overline{u})+\sum_{v\notin\overline{u}}\tau'(v,\overline{u})-\varphi_n(\mathbf{t}_0)\right|\\
 &\leq\sum_{v\in\overline{u}}|\tau'(v,\overline{u})-\varphi'_n(\atomictypeof(\overline{u},v))|\\
 &\quad\quad\quad+\biggl|\sum_{v\notin\overline{u}}\sum_{\mathbf{t}\in\atomictypes_{\mathrm{out}}[\mathbf{t}_0]}\mathbbm{1}_{\atomictypeof(\overline{u},v)=\mathbf{t}}\tau'(v,\overline{u})-n\sum_{\mathbf{t}\in\atomictypes_{\mathrm{out}}[\mathbf{t}_0]}\varphi'_n(\mathbf{t})\atppercent(\mathbf{t})\biggr|\\
 &\leq\sum_{v\in\overline{u}}\frac{\lambda\ln n}{\sqrt{n}}\varphi'_n(\atomictypeof(\overline{u},v))+\sum_{\mathbf{t}\in\atomictypes_{\mathrm{out}}[\mathbf{t}_0]}\left|n\varphi'_n(\mathbf{t})\atppercent(\mathbf{t})-\sum_{v:\,\atomictypeof(\overline{u},v)=\mathbf{t}}\tau'(v,\overline{u})\right|\\
  &\leq\sum_{v\in\overline{u}}\frac{\lambda\ln n}{\sqrt{n}}\varphi'_n(\atomictypeof(\overline{u},v))+\sum_{\mathbf{t}\in\atomictypes_{\mathrm{out}}[\mathbf{t}_0]}\left(\frac{\lambda\ln n}{\sqrt{n}}+\frac{2(\ln n)^{0.9}}{\sqrt{n}}\right)n\varphi'_n(\mathbf{t})\atppercent(\mathbf{t})\\
 &\leq\frac{2\lambda\ln n}{\sqrt{n}}\varphi_n(\mathbf{t}_0).
\end{align*}

Moreover,  with probability $1-n^{-\omega(1)}$, 
\begin{align*}  \tau(\overline{u})=\sum_v\tau'(\overline{u},v)\geq\sum_{v:\,\varphi'_n(\atomictypeof(\overline{u},v))\neq 0}\tau'(\overline{u},v)\geq n^{-K'}
\end{align*}
and 
\begin{align*}
\varphi_n(\mathbf{t}_0)=n\sum_{\mathbf{t}\in\atomictypes_{\mathrm{out}}[\mathbf{t}_0]:\,\varphi'_n(\mathbf{t})\atppercent(\mathbf{t})\neq 0}\varphi'_n(\mathbf{t})\atppercent(\mathbf{t})+\sum_{\mathbf{t}\in\atomictypes_{\mathrm{in}}[\mathbf{t}_0]:\,\varphi'_n(\mathbf{t})\neq 0}\varphi'_n(\mathbf{t})\geq n^{-K'},
\end{align*}
as required.

In all the considered cases, a.s.
$$
 \tau(\overline{u})=\sum_v\tau'(v,\overline{u})\leq\sum_v n^{K'}=n^{K'+1}.
$$

Finally, let $\tau \in \Agglang_{\functionclasspoly}$. We assume that, for every $\mathbf{t}\in\atomictypes[\mathbf{t}_0]$ such that $\varphi'_n(\mathbf{t})\not\equiv 0$, $\varphi'_n(\mathbf{t})=(c(\mathbf{t})+o(1))n^{\gamma(\mathbf{t})}$ for some positive $c(\mathbf{t})$ and real $\gamma(\mathbf{t})$. The case of $\varphi'_n(\mathbf{t})\equiv 0$ for all $\mathbf{t}\in\atomictypes[\mathbf{t}_0]$ is discussed above. Therefore, it remains to prove that if $\varphi'_n(\mathbf{t})\not\equiv 0$ for some $\mathbf{t}\in\atomictypes[\mathbf{t}_0]$, then  there exist constant $c>0$ and $\gamma\in\mathbb{R}$ such that $\varphi_n(\mathbf{t})=(c+o(1))n^{\gamma}$. Again, we do not need consider the case when, for every $\mathbf{t}\in\atomictypes_{\mathrm{out}}[\mathbf{t}_0]$, $\atppercent(\mathbf{t})\varphi'_n(\mathbf{t})\equiv 0$, since it reduces to the function application. Therefore, we can assume that $\atppercent(\mathbf{t})\varphi'_n(\mathbf{t})>\atppercent(\mathbf{t})n^{-K'}>0$ for some $\mathbf{t}\in\atomictypes_{\mathrm{out}}[\mathbf{t}_0]$. In this case, $\varphi_n(\mathbf{t}_0)$ is defined in~\eqref{eq:phi-sum}, which gives
\begin{align*}
\varphi_n(\mathbf{t}_0)&=n\sum_{\mathbf{t}\in\atomictypes_{\mathrm{out}}[\mathbf{t}_0]:\,\varphi'_n(\mathbf{t})>0}\atppercent(\mathbf{t})(c(\mathbf{t})+o(1))n^{\gamma(\mathbf{t})}+\sum_{\mathbf{t}\in\atomictypes_{\mathrm{in}}[\mathbf{t}_0]:\,\varphi'_n(\mathbf{t})>0}
(c(\mathbf{t})+o(1))n^{\gamma(\mathbf{t})}\\
&=(c+o(1))n^{\gamma},
\end{align*}
where  $\gamma$ is the maximum of 
$(\gamma(\mathbf{t})+1)$, over $\mathbf{t}\in\atomictypes_{\mathrm{out}}[\mathbf{t}_0]$ such that $\varphi'_n(\mathbf{t})>0$, and $\gamma(\mathbf{t})$, over $\mathbf{t}\in\atomictypes_{\mathrm{in}}[\mathbf{t}_0]$ such that $\varphi'_n(\mathbf{t})>0$,
and 
\begin{align*}
    c=\sum_{\mathbf{t}\in\atomictypes_{\mathrm{out}}[\mathbf{t}_0]:\,\varphi'_n(\mathbf{t})>0,\,\gamma(\mathbf{t})+1=\gamma}c(\mathbf{t})\atppercent(\mathbf{t})+\sum_{\mathbf{t}\in\atomictypes_{\mathrm{in}}[\mathbf{t}_0]:\,\varphi'_n(\mathbf{t})>0,\,\gamma(\mathbf{t})=\gamma}c(\mathbf{t}).
\end{align*}

\paragraph{$\supl$ approximation.} $\tau(\overline{u})=\max_v\tau'(\overline{u},v)$. By the induction hypothesis, there exist constant $K'>0$ and a map $\varphi'_n:\atomictypes_{k+1}\to\mathbb{R}_{\geq 0}$ as in the case of sum approximation. We immediately get that 
$$
\tau(\overline{u})=\max_v\tau'(\overline{u},v)\leq n^{K'}\quad\text{ a.s.}
$$
Moreover, the event $\tau(\overline{u})\neq 0$ implies there exists $v$ such that $\tau'(\overline{u},v)\neq 0$ which, in turn, implies $\tau'(\overline{u},v)\geq n^{-K'}$, and the latter event implies $\tau(\overline{u})\geq n^{-K}$. So,
$$
\mathbb{P}(\tau(\overline{u})\geq n^{-K'}\mid \tau(\overline{u})\neq 0)=1.
$$
As above, we fix $\mathbf{t}_0\in\atomictypes_k$ and assume $\atomictypeof(\overline{u})=\mathbf{t}_0$. Let $\atomictypes[\mathbf{t}_0]$ be the set of atomic types achievable by $(v,\overline{u})$, with $\mathrm{type}(\overline{u})=\mathbf{t}_0$, and let the sets $\atomictypes_{\mathrm{in}}[\mathbf{t}_0]$ and $\atomictypes_{\mathrm{out}}[\mathbf{t}_0]$ be defined as in the previous case. For $\mathbf{t}\in\atomictypes_{\mathrm{out}}[\mathbf{t}_0]$, recall $\atppercent(\mathbf{t})$, whose existence is asserted by Lemma~\ref{lem:extensionsstabilizedense}. If, for every $\mathbf{t}\in\atomictypes_{\mathrm{in}}[\mathbf{t}_0]$, $\varphi'_n(\mathbf{t})\equiv 0$, and, for every $\mathbf{t}\in\atomictypes_{\mathrm{out}}[\mathbf{t}_0]$, $\atppercent(\mathbf{t})\varphi'_n(\mathbf{t})\equiv 0$, then $\tau(\overline{u})=0$ with probability $1-n^{-\omega(1)}$. We set $\varphi_n(\mathbf{t}_0)\equiv 0$ in this case, which is one of the possible scenarios in the assertion of Lemma ~\ref{lem:dense_main}. 

Assume  $\varphi'_n(\mathbf{t})\not\equiv 0$ for some $\mathbf{t}\in\atomictypes_{\mathrm{in}}[\mathbf{t}_0]$ or $\atppercent(\mathbf{t})\varphi'_n(\mathbf{t})\not\equiv 0$  for some $\mathbf{t}\in\atomictypes_{\mathrm{out}}[\mathbf{t}_0]$. We further assume that, for every type $\mathbf{t}\in\atomictypes_{\mathrm{out}}[\mathbf{t}_0]$, $\atppercent(\mathbf{t})\neq 0$. Let 
$$
\varphi_n:=\max_{\mathbf{t}\in\atomictypes[\mathbf{t}_0]}\varphi'_n(\mathbf{t})>n^{-K'}.
$$
By the induction hypothesis and the union bound, with probability $1-n^{-\omega(1)}$, for every $\mathbf{t}\in\atomictypes[\mathbf{t}_0]$ and every $v\in [n]$, the term $\tau'(\overline{u},v)$ satisfies
$
|\tau'(\overline{u},v)-\varphi'_n(\mathbf{t})|\leq\frac{\lambda\ln n}{\sqrt{n}}\cdot \varphi'_n(\mathbf{t})$, if $\atomictypeof(\overline{u},v) =\mathbf{t}$.
 Therefore, with probability $1-n^{-\omega(1)}$,
\begin{align*}
 \tau\geq\max_{v\in[n]}\varphi'_n(\atomictypeof(\overline{u},v))\left(1-\lambda\cdot\frac{\ln n}{\sqrt{n}}\right)=\varphi_n\left(1-\lambda\cdot\frac{\ln n}{\sqrt{n}}\right)
\end{align*}
and
\begin{align*}
 \tau\leq\max_{v\in[n]}\varphi'_n(\atomictypeof(\overline{u},v))\left(1+\lambda\cdot\frac{\ln n}{\sqrt{n}}\right)=\varphi_n\left(1+\lambda\cdot\frac{\ln n}{\sqrt{n}}\right).
\end{align*}

Finally, if $\tau \in \Agglang_{\functionclasspoly}$, then, applying the induction hypothesis and the definition of $\varphi_n$, we immediately get that $\varphi_n=(c+o(1))n^{\gamma}$ for some $c>0$ and $\gamma\in\mathbb{R}$, completing the proof.

\end{proof}

\section{Sparse irrational case}
\label{sec:sparse-all}

We now consider {\erdosrenyi} random graphs with  
$$
p=p(n)=n^{-\alpha},\, \text{ where $0<\alpha<1$ is irrational.} 
$$
The value of $\alpha$ will be fixed throughout this section. Let $\mathbf{G}_n\sim G(n,p)$.
In this section, we prove an analogue of Theorems~\ref{th:dense-concentration}~and~\ref{th:dense-limits} in the sparse regime.

\begin{theorem}[Concentration for sparse random graphs]
\label{th:sparse-concentration-and-limits}
The conclusions of Theorem~\ref{th:dense-concentration} and Theorem~\ref{th:dense-limits} hold for $\mathbf{G}_n\sim G(n,n^{-\alpha})$ as well.
\end{theorem}

Similarly, Theorem~\ref{th:sparse-concentration-and-limits} implies concentration for the number of subgraphs and the maximum number of MRGs.
\begin{corollary}
\label{cor:sparse-cor}
The conclusions of Corollary~\ref{cor:dense-cor} hold for $\mathbf{G}_n\sim G(n,n^{-\alpha})$ as well.
\end{corollary}

In particular, the number of triangles in $\mathbf{G}_n$ is concentrated around $\frac{1}{6}n^{3-3\alpha}$ $\whp$ Similarly to the dense case, concentration results for the number of subgraphs and maximum number of MRGs, extending a set of roots,  were known, see~\cite{smallsubgraphs,Spencer_count,Warnke}. Nevertheless, our proof technique allows to derive a finer result for extensions 
 which, in this generality\footnote{To our knowledge, this result is only known for so called ``strictly balanced'' 
  MRGs~\cite{Warnke}.}, is not stated elsewhere: letting $\xi:=\max_U X_{R,H}(U)$, there exists $\varepsilon>0$ such that, $|\xi/\mathbb{E}\xi-1|\leq n^{-\varepsilon}$ $\whp$

\paragraph{Proof strategy.} We prove Theorem~\ref{th:sparse-concentration-and-limits} following the strategy in the proof of Theorems~\ref{th:dense-concentration},~\ref{th:dense-limits}. However, the proof is significantly more complex due to the following crucial difference between the sparse and the dense regime: in the sparse regime Lemma~\ref{lem:extensionsstabilizedense} does not hold. In other words, it is not true that any tuple of vertices can be extended 
   in $\mathbf{G}_n$ in any given way. For example, if $\alpha>\frac{1}{2}$, then most pairs of vertices do not have a common neighbour $\whp$
In order to resolve this issue, we use the tools that were developed in~\cite{shelahspencersparse} to prove the FO 0-1 law. In particular, we distinguish between sparse and dense extensions (see Definition~\ref{def:extensions}, in what follows)   and show that every tuple of vertices in $\mathbf{G}_n$ has any fixed sparse extension, but almost all tuples of vertices do not have a fixed dense extension $\whp$ (see Theorem~\ref{thm:extensionexpectation} and Corollary~\ref{cor:sparse_graphs}). Actually, these concepts allow to classify all possible extensions, since every extension can be represented as a sequence of sparse and dense extensions.  Moreover, for every fixed tuple of vertices $\overline{u}$, the number of dense extensions is bounded (see Fact~\ref{fact:closuresize}), which motivates the notion of a closure $\mathrm{cl}_s(\overline{u})$ of $\overline{u}$ that comprises all possible dense extensions of a given size $s$, see Definition~\ref{def:closure}. 
  Given a tuple and its closure, we know that it does not have any dense extensions outside the closure, while, the number of sparse extensions is large and is concentrated 
    around its expectation, as in Lemma~\ref{lem:extensionsstabilizedense}. It leads to the following refinement of the notion of an atomic type, which we call a closure type: the closure type of a tuple $\overline{u}$ is the MRG $\mathrm{cl}_s(\overline{u})[\overline{u}]$, see Definition~\ref{def:closuretype}.

\paragraph{Organisation.} The remainder of this section is organised as follows. In Subsection~\ref{subsc:sparse-proof} we formalise the above proof strategy: 
\begin{itemize}
\item First, we introduce definitions of sparse, dense extensions and closures and prove their required properties. 
\item Next, we establish Theorem~\ref{thm:extensionexpectation} --- an analogue of the stabilization Lemma~\ref{lem:extensionsstabilizedense}. 
\item Then we define closure types and state and prove Lemma~\ref{lem:sparse_main}, which plays the same role as Lemma~\ref{lem:dense_main} in the dense case --- it approximates terms and establishes concentration of every (not necessarily closed) term around $\varphi_n$. 
\end{itemize}
Finally, in Subsection~\ref{subsc:convergence-law-sparse}, we extend the results from~\cite{usneurips24,uslics} by showing that terms in the language $\Aggo{\Lipsch,\Mean,\LMean,\supl}$ satisfy the convergence law.

\subsection{Proof of Theorem~\ref{th:sparse-concentration-and-limits}}
\label{subsc:sparse-proof}

\begin{definition}[Number of extending edges or vertices, density, maximum density]
For a graph $G$ and its induced subgraph $H$, we let $e(V(H),G)$ or $e(H,G)$ denote 
$|E(G)|-|E(H)|$.
We use either $v(V(H),G)$ or $v(H,G)$ to denote $|V(G)|-|V(H)|$.
We also let $\rho(G)=\frac{|E(G)|}{|V(G)|}$ denote the {\it density of $G$} and $\rho^{\max}(G)=\max_{H\subseteq G}\rho(H)$ denote the {\it maximum density of~$G$}.
\end{definition}

\begin{definition}[Extensions: dense, $s$-dense, and sparse] \label{def:extensions}
Let $F_1$ be an induced subgraph of a graph~$F_2$.
\begin{compactitem}
\item Graph $F_2$ is a \emph{dense extension} of $F_1$ if, for every set $V(F_1)\subseteq S\subset V(F_2)$, 
$$
e(S,F_2)>\frac{1}{\alpha}v(S,F_2).
$$
We also say $(F_1, F_2)$ is a {\it dense pair}. If a dense pair $(F_1,F_2)$ has $v(F_1,F_2)\leq s$, then we say that the pair $(F_1, F_2)$ is {\it $s$-dense}. 

\item Graph $F_2$ is a \emph{sparse extension} of $F_1$ (and $(F_1,F_2)$ is a {\it sparse pair}), if, for every set $V(F_1)\subset S\subseteq V(F_2)$, 
$$
e(F_1,F_2\upharpoonright S)<\frac{1}{\alpha}v(F_1,F_2\upharpoonright S).
$$.

\item Let $F_1$ have $V(F_1)=[k]$, let $F'_1$ be an induced subgraph of $F'_2$ with $V(F'_1)=\{u_1,\ldots,u_k\}$, and let $f:V(F_2)\to V(F'_2)$ be an isomorphism of graphs $F_2\setminus E(F_1)$ and $F'_2\setminus E(F'_1)$ such that $f(i)=u_i$ for every $i\in[k]$ (that is, $f$ is an isomorphism of MRGs $(F_2\setminus E(F_1))[(1,\ldots,k)]$ and $(F'_2\setminus E(F'_1))[(u_1,\ldots,u_k)]$). 
 Then we say that $F'_2$ is an {\it $(F_1,F_2)$-extension of $F'_1$} (or sometimes we say that that $F'_2$ is an $(F_1,F_2)$-extension of $V(F'_1)$ or a $(V(F_1),F_2)$-extension).
\end{compactitem}
\end{definition}

The following fact is the direct corollary of the definition of a dense pair.
\begin{fact}
\label{fact:dense-sub-extension}
If pair $(F_1,F_2)$ is dense and $F'_1\supset F_1$ is such that $V(F_2)\setminus V(F'_1)\neq\varnothing$, then the pair $(F'_1,F'_1\cup F_2)$ is also dense.
\end{fact}

\begin{definition}[Closure] \label{def:closure}
Let $G$ be a graph on $[n]$.
Let $k$ be a positive integer, let $\overline{u}\in ([n])_k$, and let $s$ be a positive integer. The {\it $s$-closure} of $\overline{u}$ in $G$, denoted $\closure^G_s(\overline{u})$ is an induced subgraph  $H$ such that
\begin{compactitem}
\item $\overline{u}\subseteq V(H)$;
\item $H$ does not have $s$-dense extensions in $G$;
\item there exists a sequence of graphs $G\upharpoonright\overline{u}= H_1\subset\ldots\subset H_r=H$ such that each pair $(H_i,H_{i+1})$ is $s$-dense.
\end{compactitem}
\end{definition}
We let $\closure_0^G(\overline{u})=G\upharpoonright\overline{u}$. We note that the closure is well-defined, see~\cite[Section 4.2]{Strange}.

We shall use two facts about closures of subgraphs of other closures.

\begin{fact}
\label{fact:closure-inclusion}
Let $F'\subset F$ be subgraphs of a graph $G$. Let $s'\leq s$ be positive integers. Then $\closure_{s'}^{G}(F')\subseteq \closure_{s}^{G}(F)$.
\end{fact}

Fact~\ref{fact:closure-inclusion} follows immediately from the definition of a closure since any sequence of dense extensions of $F'$ is also a sequence of dense extensions of $F$.

\begin{fact}
\label{fact:closure-intersection}
Let $F$ be a subgraph of a graph $G$, let $s,s'$ be positive integers, and let $F\subset F'\subseteq\closure^G_s(F)$. If $\closure^G_{s'}(F')$ has at most $s$ vertices, then $\closure_{s'}^{G}(F')\subseteq \closure_{s}^{G}(F)$.
\end{fact}

Fact~\ref{fact:closure-intersection} follows from the observation that ``subextensions'' 
 of a dense extension are also dense. Indeed, if $F'= H_1\subset\ldots\subset H_r=\closure_{s'}^{G}(F')$ is a sequence of $s'$-dense extensions that constitutes $\closure_{s'}^{G}(F')$ and if $H_i$ is the first graph in this sequence that is not entirely inside $\closure_{s}^{G}(F)$, then the pair $(\closure_{s}^{G}(F),\closure_{s}^{G}(F)\cup H_i)$ is dense due to Fact~\ref{fact:dense-sub-extension} --- a contradiction.

We also need the fact that closures in random graphs are bounded $\whp$, see the proof in~\cite[Theorem 4.3.2]{Strange}.

\begin{fact} \label{fact:closuresize}
There exists a coordinate-wise increasing function $\ell:\mathbb{Z}^2_{>0}\to\mathbb{Z}_{>0}$ such that, for every $k,s\in\mathbb{Z}_{>0}$, $\whp$, for every $\overline{u}\in([n])_k$,
$$
\left|V(\closure^{\mathbf{G}_n}_s(\overline{u}))\right|\leq \ell(k,s).
$$
\end{fact}

Finally, the following fact claiming that $\whp$ $\mathbf{G}_n$ does not have dense sugrbaphs is a simple corollary of Markov's inequality --- see~\cite[Theorem 3.4]{Janson}.

\begin{fact}
\label{fact:dense-graphs}
Let $F$ be a graph with $\rho(F)>1/\alpha$. Then $\whp$ $\mathbf{G}_n$ does not contain a copy of $F$.
\end{fact}

In what follows, when $G=\mathbf{G}_n$, we omit the superscript ${\mathbf{G}_n}$ and write $\closure_s(\overline{u})=\closure^{\mathbf{G}_n}_s(\overline{u})$.

\begin{definition}[Number of extensions] \label{def:numberofxextensions}
Fix a pair of graphs $F_1\subset F_2$ with $k:=|V(F_1)|$ and an integer $s>0$. Let $\overline{u}\in ([n])_{k}$ and  $U\subseteq \overline{u}$.

Let $X_{(F_1,F_2;s)}(\overline{u};U)$ be the number of $(F_1,F_2)$-extensions $F'_2$ of $\overline{u}$ in $\mathbf{G}_n$ such that $F'_2\setminus U$ does not have $s$-dense extensions $F'_3$ in $\mathbf{G}_n$ with $V(F'_3)\setminus V(F'_2)\neq\varnothing$ and 
$$
 E(F'_3)\cap((V(F'_2)\setminus \overline{u})\times(V(F'_3)\setminus V(F'_2)))\neq\varnothing.
 $$ 
\end{definition}

The following key result is an analog of Lemma \ref{lem:extensionsstabilizedense} to the case of sparse random graphs:
\begin{theorem}[Extension counts]\label{thm:extensionexpectation}
Let $0\leq k'\leq k$ be integers.
For any sparse $(F_1,F_2)$ and any $s\in\mathbb{Z}_{\geq 0}$ (here, $V(F_1)$ can be empty), we have that  
$$
\mu(F_1,F_2;s) := \mathbb{E}X_{(F_1,F_2;s)}(\overline{u};U) 
$$ 
does not depend on $\overline{u}\in ([n])_{|V(F_1)|}$ and $U\subseteq\overline{u}$ of size $k'$ and equals
\begin{multline}
(1+o(1))\cdot\frac{(v(F_1,F_2))!}{\mathrm{aut}(F_1,F_2)}\cdot{n-|V(F_1)|\choose v(F_1,F_2)}p^{e(F_1,F_2)}
=(1/\mathrm{aut}(F_1,F_2)+o(1))\cdot n^{v(F_1,F_2)-\alpha \cdot e(F_1,F_2)}
\label{eq:exetnsions-expectation-asymp}
\end{multline}
where $\mathrm{aut}(F_1,F_2)$ is the number of automorphisms of $F_2$ that preserve each vertex of $F_1$. Moreover, there exists $\varepsilon>0$ such that,
with probability $1-n^{-\omega(1)}$, every $\overline{u}\in ([n])_{|V(F_1)|}$ and every $U\subseteq\overline{u}$ of size $k'$ satisfies
$$
 |X_{(F_1,F_2;s)}(\overline{u};U)/\mu(F_1,F_2;s)-1|\leq n^{-\varepsilon}.
$$
\end{theorem}

\begin{proof}
For a pair of graphs $F_1\subset F_2$, denote
\begin{align*}
 \rho(F_1,F_2)=\frac{e(F_1,F_2)}{v(F_1,F_2)},\quad\text{ and }\quad 
 \rho^{\max}(F_1,F_2)=\max_{V(F_1)\subset S\subseteq V(F_2)}\rho(F_1,F_2\upharpoonright S).
\end{align*}
Let us say that a pair of graphs $(F_1,F_2)$ is {\it strictly balanced}, if $F_1\subset F_2$ and, for every $V(F_1)\subset S\subset V(F_2)$, 
$$
 \rho(F_1,F_2\upharpoonright S)<\rho(F_1,F_2)=\rho^{\max}(F_1,F_2).
$$
Fix any pair $(F_1,F_2)$. Let us observe that there exists a sequence $F_1=H_1\subset H_2\subset\ldots\subset H_r=F_2$ such that every pair $(H_{i-1},H_i)$ is strictly balanced, and, for every $i$, $\rho(H_{i-1},H_i)\geq\rho(H_i,H_{i+1})$. Indeed, for every $i$, assuming that $H_{i-1}$ has been already defined, let $H_i$ be an inclusion-minimal graph with the maximum density $\rho(H_{i-1},H_i)$ such that  $H_{i-1}\subset H_i\subseteq F_2$. Then, clearly, all pairs $(H_{i-1},H_i)$ are strictly balanced. Assume for some $i$, $\rho(H_{i-1},H_i)<\rho(H_i,H_{i+1})$. Then 
\begin{align*}
\rho(H_{i+1},H_{i-1})=\frac{e(H_{i+1},H_i)+e(H_i,H_{i-1})}{v(H_{i+1},H_i)+v(H_i,H_{i-1})}
>\min\{\rho(H_{i+1},H_i),\rho(H_i,H_{i-1})\}=\rho(H_i,H_{i-1})
\end{align*}
--- a contradiction with the definition of $H_i$.

For a pair of graphs $F_1\subset F_2$ and $\overline{u}\in([n])_{|V(F_1)|}$, let  $X^{all}_{(F_1,F_2)}(\overline{u})$ 
denote the number of $(F_1,F_2)$-extensions of $\overline{u}$ in $\mathbf{G}_n$. Clearly, $\mathbb{E}X^{all}_{(F_1,F_2)}(\overline{u})$ does not depend on $\overline{u}$ and equals
$$
 \frac{(v(F_1,F_2))!}{\mathrm{aut}(F_1,F_2)}{n-k\choose v(F_1,F_2)}p^{e(F_1,F_2)} 
$$
Let $\mu^{all}(F_1,F_2)$ denote the quantity above. 

We will rely on a slightly refined version of~\cite[Theorem 5]{Spencer_count}, which is presented below.\footnote{In the original version, it is stated that the event $(1-\varepsilon)\mu^{all}(F_1,F_2)<X_{(F_1,F_2)}(\overline{u})<(1+\varepsilon)\mu^{all}(F_1,F_2)$ holds with probability $1-n^{-\omega(1)}$ for every $\varepsilon>0$. This weaker version holds true for $p$ that can be a polylogarithmic factor close to the threshold probability $n^{-1/\rho(F_1,F_2)}$. In our case, $p$ is $n^{\Theta(1)}$-far from the treshold and exactly the same proof gives the refined result. Alternatively, Fact~\ref{fact:ext-count} follows directly from the Kim-Vu concentration inequality~\cite{KimVu}.}

\begin{fact}
\label{fact:ext-count}
Let $(F_1,F_2)$ be a strictly balanced pair with $\rho(F_1,F_2)<1/\alpha$. Let $k:=|V(F_1)|$. There exists $\varepsilon>0$ such that, for every $\overline{u}\in([n])_k$, with probability $1-n^{-\omega(1)}$, 
\begin{equation}
\label{eq:spencer_concentration}
(1-n^{-\varepsilon})\mu^{all}(F_1,F_2)<X^{all}_{(F_1,F_2)}(\overline{u})<(1+n^{-\varepsilon})\mu^{all}(F_1,F_2).
\end{equation}
\end{fact}

Now, let $(F_1,F_2)$ be sparse, $|V(F_1)|=k$, and let $F_1=H_1\subset H_2\subset\ldots\subset H_r=F_2$ be a sequence of graphs   
as above. Since $(F_1,F_2)$ is sparse, we get that, for every $i$, 
$$
\rho(H_{i-1},H_i)\leq\rho(H_1,H_2)<\frac{1}{\alpha}.
$$
By Fact~\ref{fact:ext-count} applied $r-1$ times, we get that~\eqref{eq:spencer_concentration} holds with probability $1-n^{-\omega(1)}$, for every $\overline{u}\in([n])_k$.

We have proved that the total number of all $(F_1,F_2)$-exten\-sions satisfies the conclusion of Theorem~\ref{thm:extensionexpectation}: $\mathbb{E}X^{all}_{(F_1,F_2)}(\overline{u})$ does not depend on $\overline{u}$, equals~\eqref{eq:exetnsions-expectation-asymp}, and $X^{all}_{(F_1,F_2)}(\overline{u})$ has the desired concentration property~\eqref{eq:spencer_concentration}. In order to complete the proof, we need to prove that $\mu(F_1,F_2;s)=(1+o(1))\mu^{all}(F_1,F_2)$ and that the number of $(F_1,F_2)$-extensions that have $s$-dense extensions is at most $\mu^{all}(F_1,F_2)\cdot n^{-\Theta(1)}$, with probability $1-n^{-\omega(1)}$.

Let $F_3$ be a graph such that the inclusion $F_2\subset F_3$ is proper, $(F_2,F_3)$ is $s$-dense, and there exists at least one edge between $V(F_3)\setminus V(F_2)$ and $V(F_2)\setminus V(F_1)$ in $F_3$. 
 In order to complete the proof, it suffices to prove that, for some $\varepsilon'>0$ and for every $\overline{u}\in([n])_k$, with probability $1-n^{-\omega(1)}$, 
\begin{equation}
\label{eq:ext-count-dense-are-rare}
X^{all}_{(F_1,F_3)}(\overline{u})<\mu^{all}(F_1,F_2)/n^{\varepsilon'}.
\end{equation}
We note that 
$$
 v(F_1,F_3)-\alpha\cdot e(F_1,F_3)=(v(F_1,F_2)-\alpha\cdot e(F_1,F_2))+
 (v(F_2,F_3)-\alpha\cdot e(F_2,F_3))<v(F_1,F_2)-\alpha\cdot e(F_1,F_2).
$$
Let $F_1=H'_1\subset H'_2\subset\ldots\subset H'_{r'}=F_3$ be a sequence of strictly balanced extensions such that $\rho(H'_{i-1},H'_i)\geq\rho(H'_i,H'_{i+1})$ for all $i$.  If $\rho(H'_1,H'_2)<1/\alpha$, the result follows from Fact~\ref{fact:ext-count}.

Assume $\rho(H'_1,H'_2)>1/\alpha$. Let $H'$ be an inclusion-minimal graph such that $F_1\subset H'\subseteq F_3$ and the pair $(H',F_3)$ is sparse. Note that the pair $(F_1,H')$ is dense since, otherwise, for an inclusion-maximal subgraph $F_1\subseteq H''\subset H'$ such that $\rho(H'',H')<1/\alpha$, the pair $(H'',H')$ is sparse, meaning that the pair $(H'',F_3)$ is sparse as well --- a contradiction with minimality of $H'$. Then
\begin{align*}
 v(H',F_3)-\alpha e(H',F_3)
 &=(v(H',H'\cup F_2)-\alpha e(H',H'\cup F_2))
 +(v(H'\cup F_2,F_3)-\alpha e(H'\cup F_2,F_3))\\
 &\leq
 (v(H'\cap F_2,F_2)-\alpha e(H'\cap F_2,F_2))+(v(H'\cup F_2,F_3)-\alpha e(H'\cup F_2,F_3)).
\end{align*}
Since $(F_2,F_3)$ is dense, we get that $v(H'\cup F_2,F_3)-\alpha\cdot e(H'\cup F_2,F_3)\leq 0$ and the equality is achieved only when $H'\cup F_2=F_3$. Since $(F_1,F_2)$ is sparse, we get
\begin{align*}
 v(H',F_3)-\alpha\cdot e(H',F_3)&\leq 
 v(H'\cap F_2,F_2)-\alpha\cdot e(H'\cap F_2,F_2)\\
 &=
 (v(F_1,F_2)-\alpha\cdot e(F_1,F_2))-(v(F_1,H'\cap F_2)-\alpha\cdot e(F_1,H'\cap F_2))\\
 &\leq v(F_1,F_2)-\alpha\cdot e(F_1,F_2)
\end{align*}
where the second equality is achieved only when $F_1=H'\cap F_2$. So, we have both equalities only when $H'=F_3\upharpoonright (V(F_1)\cup V(F_3\setminus V(F_2))$. But then the equality $v(H',F_3)-\alpha\cdot e(H',F_3)=v(F_1,F_2)-\alpha\cdot e(F_1,F_2)$ is impossible by the assumption that there are edges in $F_3$ between $F_3\setminus F_2$ and $F_2\setminus F_1$, since $v(F_1,F_2)=v(H',F_3)$. 

We conclude that
\begin{equation}
\label{eq:density-comparison}
 v(H',F_3)-\alpha\cdot e(H',F_3)< v(F_1,F_2)-\alpha\cdot e(F_1,F_2).
\end{equation}

Let $\varepsilon''>0$. Recall, from the comments above, after the definition of $H'$, that the pair $(F_1,H')$ is dense. If $\overline{u}$ has at least $n^{\varepsilon''}$ $(F_1,H')$-extensions, then, for a large enough $C>0$, we can find at least $n^{\varepsilon''/C}$ extensions $G'_1,\ldots,G'_m$ of $\overline{u}$, such that each $G'_i$ is not entirely inside $G'_1\cup\ldots\cup G'_{i-1}$. Then, for a small enough constant $\delta=\delta(F_1,H')>0$ that does not depend on $\varepsilon''$ (so, $\varepsilon''$ can be chosen significantly smaller than $\delta$), the graph $G'=G'_1\cup\ldots\cup G'_m$ has $v=\Theta(n^{\varepsilon''/C})$ vertices and at least $(v-k)(1/\alpha+\delta)$ edges. By the union bound, the probability of this event is at most
\begin{align*}
\sum_{v=\Theta(n^{\varepsilon''/C})}{n\choose v}{v^2\choose (v-k)(1/\alpha+\delta)}n^{-\alpha(v-k)(1/\alpha+\delta)}
 &\leq \sum_{v=\Theta(n^{\varepsilon''/C})}n^{v-(1-\varepsilon''/\alpha)((1+\delta\alpha)v-k(1+\delta\alpha))}\\
 &=n^{-\omega(1)}.
\end{align*}
This completes the proof of~\eqref{eq:ext-count-dense-are-rare} in the case $H'=F_3$. 
 If $H'\subset F_3$ is a proper inclusion, we let $\mathcal{E}$ be the set of all $(F_1,H')$-extensions of $\overline{u}$. We may assume that $|\mathcal{E}|<n^{\varepsilon''}$. Since $(H',F_3)$ is sparse, by Fact~\ref{fact:ext-count}, with probability $1-n^{-\omega(1)}$, for every $G'\in\mathcal{E}$, the number of $(H',F_3)$-extensions of $G'$ is $O(n^{v(H',F_3)-\alpha\cdot e(H',F_3)})$. Due to~\eqref{eq:density-comparison}, there exists a postive constant $\delta'\gg\varepsilon''$ such that, with probability $1-n^{-\omega(1)}$, the total number of $(F_1,F_3)$-extensions of $\overline{u}$, is at most
\begin{align*} O\left(\sum_{H'\in\mathcal{E}'}n^{v(H',F_3)-\alpha\cdot e(H',F_3)}\right)=O\left(n^{\varepsilon''}\cdot n^{v(F_1,F_2)-\alpha\cdot e(F_1,F_2)-\delta'}\right)
<\mu^{all}(F_1,F_2)\cdot n^{-\delta'/2},
\end{align*}
completing the proof.

\end{proof}

In the case $V(F_1)=\varnothing$, Theorem ~\ref{thm:extensionexpectation} implies that the number of graphs isomorphic to $F$ with maximum density less than $1/\alpha$ that do not have $s$-dense extensions 
  satisfies the same concentration inequality.
\begin{corollary}
\label{cor:sparse_graphs}
For any graph $F$ with $\rho^{\max}(F)<1/\alpha$ and any $s\in\mathbb{Z}_{\geq 0}$,  the number of graphs isomorphic to $F$ that do not have $s$-dense extensions is concentrated around its expectation, and its expectation  is
$$
 \frac{(|V(F)
|)!}{\mathrm{aut}(F)}\cdot{n\choose |V(F)
|}p^{|E(F)|}
=(1/\mathrm{aut}(F)+o(1))\cdot n^{|V(F)|-\alpha \cdot |E(F)|},
$$
where $\mathrm{aut}(F)$ is the number of automorphisms of $F$. 
\end{corollary}

\begin{proof}
It suffices to show that the pair $(\varnothing,F)$ is sparse, i.e., for every $S\subseteq V(F)$, we have that $|E(F \upharpoonright S)|<\frac{1}{\alpha}|S|$. This immediately follows from 
$$
\frac{1}{\alpha}>\rho^{\max}(F)\geq\rho(F\upharpoonright S)=\frac{|E(F\upharpoonright S)}{|S|}.
$$
\end{proof}

\begin{definition}[Closure types]\label{def:closuretype}
An {\it $s$-closure type} is a MRG  $F[(1,\ldots,k)]$ for some positive integer $k$, where $F=\closure_s^{F}([k])$.

Let $k$ be a positive integer, let $G$ be a graph on $[n]$, let $\overline{u}=(u_1,\ldots,u_k)\in([n])_k$, and let $F[(1,\ldots,k)]$ be an $s$-closure type. We say that $\overline{u}$ {\it has $s$-closure type $F[(1,\ldots,k)]$ in $G$}, if MRGs $F[(1,\ldots,k)]$ and $\closure_s^G{(\overline{u})}[\overline{u}]$ are isomorphic.
\end{definition}

In other words, a closure type is the isomorphism class of a MRG $F[(1,\ldots,k)]$.

\begin{definition} [$\closuretypes^{(s)}$ and $\closuretypes^{(s)}_k$]\label{def:notsodenseclosures}
 Let $\closuretypes^{(s)}$ be the set of {\it all}
 \footnote{Every isomorphism class $F[\overline{x}]$ appears exactly once in $\closuretypes^{(s)}$.} 
 $s$-closure types $F[\overline{x}]$ such that 
 \begin{compactitem}
 \item $F$ has maximum density strictly less than $1/\alpha$ and 
 \item $|V(F)|\leq \ell(|V(F)|,s)$, where $\ell$ is the function from Fact ~\ref{fact:closuresize}.
 \end{compactitem}
Let $\closuretypeof^G_s:\cup_k([n])_k\to\closuretypes^{(s)}$ map tuples to their $s$-closure types in $G$. We let $\closuretypes^{(s)}=\cup_{k=0}^{\infty}\closuretypes^{(s)}_k$, where $\closuretypes^{(s)}_k$ is the set of all $s$-closure types $F[(1,\ldots,k)]$. 
\end{definition}

For example, let $\alpha\in(1/2,1)$, let $G$ be a graph with vertices $u_1,u_2,v$ such that $v$ is a common neighbour of $u_1,u_2$, and any vertex outside of $\{u_1,u_2,v\}$ has at most one neighbour in this set. Then $\closuretypeof_1^G(u_1,u_2)=F[(1,2)]$, and the graph $F$ has $V(F)=\{1,2,x\}$, where $x$ is adjacent to both 1 and 2.

From Corollary~\ref{cor:sparse_graphs} it follows that for every $s$ and every $\mathbf{t}\in\closuretypes^{(s)}_k$, $\whp$ there exists $\overline{u}\in([n])_k$ with $s$-closure type $\mathbf{t}$ in $\mathbf{G}_n$. On the other hand, by Fact ~\ref{fact:closuresize} and Fact~\ref{fact:dense-graphs}, for every positive integer $k$ and $s$, $\whp$ every $k$-tuple $\overline{u}\in([n])_k$ has $s$-closure type that belongs to $\closuretypes_k^{(s)}$. In other words, $\closuretypes_k^{(s)}$ is exactly the set of all $s$-closure types achievable by $k$-tuples in $\mathbf{G}_n$ $\whp$ In what follows, when $G=\mathbf{G}_n$, we omit the superscript ${\mathbf{G}_n}$ and write $\closuretypeof_s=\closuretypeof^{\mathbf{G}_n}_s$.

We extend the definition of closure types from tuples of different vertices to $\overline{u}\in[n]^k$ in a natural way.

In contrast to the dense case, the notion of type depends on a number $s$, and
we need to determine $s$ from the depth of the term. Our next definition introduces a sequence of values of $s$, where $s_i$ corresponds to a term of depth $D-i$, for some large integer parameter $D$ that we fix in advance. 
\begin{definition}[Sufficiently rapid sequence] \label{def:dfrapid}
Let $D$ and $W$ be positive integers. A \emph{$(D,W)$-rapid} sequence of integers  is  a $s_0\geq s_1\geq\ldots\geq s_D=0$ such that 
$$
 s_{i-1}\geq \ell(W,s_i)
$$ 
for all $i\in[D]$.
\end{definition}

We now fix arbitrary positive integers $D$ and $W$. Let $s_0 \ldots s_D$ be a $(D,W)$-rapid sequence. 
\begin{definition}[$\mathcal{L}$] 
We let $\mathcal{L}$ be the event from Fact~\ref{fact:closuresize}: for every $k\leq F$, every $s\leq s_0$, and every $\overline{u}\in([n])_k$, $|V(\mathrm{cl}_s(\overline{u}))|\leq\ell(k,s)$. Due to Fact~\ref{fact:closuresize}, $\mathcal{L}$ holds $\whp$
\end{definition}

\begin{lemma}[Inductive Term Approximation with $\Sigma$: sparse case] \label{lem:sparse_main}
Let $\tau \in \Agglang_{\functionclassrellip}$ be a term of depth $d\leq D$ with $k$ free variables and at most $W$ variables in total and let $s_0\geq s_1\geq\ldots\geq s_D=0$ be a $(D, W)$-rapid sequence. There exist constants $\varepsilon,K>0$ and maps $\varphi_n:\closuretypes^{(s_{D-d})}_k\to \mathbb{R}_{\geq 0}$ for every $n \in \nats$,  such that, for every $\overline{u}\in[n]^k$ and every $\mathbf{t}\in\closuretypes_k^{(s_{D-d})}$, we have 
$$
\mathbb{P}(n^{-K}\leq\tau(\overline{u})\leq n^{K}\mid\tau(\overline{u})\neq 0)=1
$$ 
and one of the following two assertions holds:
\begin{compactitem}
\item $\varphi_n(\mathbf{t})\equiv\mathrm{0}$ and $\tau$ is typically zero for tuples of type $\mathbf{t}$:  
 $$
 \mathbb{P}(\tau(\overline{u})=\varphi_n(\mathbf{t})=0\mid\mathcal{L}\wedge\{\closuretypeof_{s_{D-d}}(\overline{u})=\mathbf{t}\})
 =1-n^{-\omega(1)},
 $$ 
 \item $\varphi_n(\mathbf{t})>n^{-K}$, and $\tau$ is concentrated around $\varphi_n(\mathbf{t})$ for tuples of type $\mathbf{t}$:
\begin{align}
\mathbb{P}\left(|\tau(\overline{u})-\varphi_n(\mathbf{t})|>n^{-\varepsilon} \varphi_n(\mathbf{t})\mid \mathcal{L}\wedge\{\closuretypeof_{s_{D-d}}(\overline{u}) =\mathbf{t}\}\right)
=n^{-\omega(1)}.
\label{eq:concentratio-sparse-main}
\end{align}

\end{compactitem}
Moreover, if $\tau \in \Agglang_{\functionclasspoly}$, and $\varphi_n(\mathbf{t})>0$, then, for some positive $c=c(\mathbf{t})$ and real $\gamma=\gamma(\mathbf{t})$, we have 
$$
\varphi_n(\mathbf{t})=(c+o(1))n^{\gamma}.
$$
\end{lemma}

The derivation of Theorem~\ref{th:sparse-concentration-and-limits} from Lemma~\ref{lem:sparse_main} is verbatim to the proof of Theorems~\ref{th:dense-concentration},~\ref{th:dense-limits}, therefore we omit it.

In order to prove Lemma~\ref{lem:sparse_main}, we need the following assertion, which was also used in~\cite{shelahspencersparse} but was not explicitly stated in that paper. For the sake of completeness, we prove this fact below.
\begin{fact}
\label{fact:pair-of-closures-sparse}
Let $s_1,s_2$ be non-negative integers, let $G$ be a graph on a vertex set $V$, and let $U_1\subset U_2\subseteq V$ be such that 
\begin{itemize}
\item $U_2\not\subset\closure^G_{s_1}(U_1)$,
\item the pair $(U_1,G\upharpoonright U_2)$ is sparse,
\item $|\mathrm{cl}^G_{s_2}(U_2)|\leq s_1$.
\end{itemize}
Then the pair
$(\closure^G_{s_1}(U_1),\closure^G_{s_2}(U_2)\cup \closure^G_{s_1}(U_1))$ is sparse.
\end{fact}

\begin{proof}
Denote $F_1:=\closure^G_{s_1}(U_1)$, $F_2:=\closure^G_{s_2}(U_2)\cup \closure^G_{s_1}(U_1)$ and assume the opposite --- there exists $S$ such that $V(F_1)\subseteq S\subset V(F_2)$ and $e(F_1,F_2\upharpoonright S)>\frac{1}{\alpha}v(F_1,F_2\upharpoonright S)$. Assume $S$ is inclusion-minimal with this property. Then, for every $V(F_1)\subset S'\subset S$, we have that $e(F_1,F_2\upharpoonright S')<\frac{1}{\alpha}v(F_1,F_2\upharpoonright S')$, implying that
$$
 e(S',F_2\upharpoonright S)=e(F_1,F_2\upharpoonright S)-e(F_1,F_2\upharpoonright S')
 >\frac{1}{\alpha}\biggl(v(F_1,F_2\upharpoonright S)-v(F_1,F_2\upharpoonright S')\biggr)
 =\frac{1}{\alpha}v(S',F_2\upharpoonright S).
$$
Therefore, the pair $(F_1,F_2\upharpoonright S)$ is dense. Since 
\begin{align*}
|S\setminus V(F_1)|\leq|V(F_2)|-|V(F_1)|\leq |V(F_2)|\leq s_1,
\end{align*} 
we get that $F_2\upharpoonright S\subset\closure_{s_1}(U_1)$ --- a contradiction with the assumption that $U_2\not\subset V(F_1)$. 
\end{proof}

\begin{proof}[Proof of Lemma~\ref{lem:sparse_main}.]

We prove the lemma by induction on the grammar for terms. We start from the atomic cases.

\paragraph{Atomic cases.}

\begin{enumerate}

\item $\tau\equiv\mathrm{const}$. In this case, the assertion is obvious for $\varphi=\tau$.

\item $\tau(u_1,u_2)=\mathbbm{1}_{\mathrm{E}(u_1,u_2)}$  is a term with two free variables. For $\mathbf{t}=F[(1,2)]\in\closuretypes_2^{s_D}$, we let $\varphi(\mathbf{t})=\mathbbm{1}_{E(F\upharpoonright [2])\neq\varnothing}$. Then $\tau(u_1,u_2)=\varphi(\mathbf{t})$ for every pair of vertices $(u_1,u_2)$ with $s_D$-closure type $\mathbf{t}$. The inductive assertion clearly follows.

\item The case of $\tau(u_1,u_2)=\mathbbm{1}_{u_1=u_2}$ is analogous to the previous edge case.

\end{enumerate}

\paragraph{Function application.}  $\tau=f(\tau_1,\ldots,\tau_m)$ is a term of depth $d$, where $f\in\functionclassrellip$. We have that each term $\tau_i$ has depth $d_i\leq d$. Let $\tau$ and $\tau_i$ have $k$ and $k_i$ free variables respectively. Let $\overline{u}\in[n]^k$. Substituting it as an input of $\tau$, identifies arguments of $\tau_1,\ldots,\tau_m$ --- we denote the respective subtuples of $\overline{u}$ by $\overline{u}_i$. For the sake of convenience and clarity of presentation, in what follows, we assume that $\overline{u}\in([n])_k$. The case when $u_i=u_j$ for some $1\leq i<j\leq k$ is verbatim.

Fix an $s_{D-d}$-closure type of $k$-tuples $\mathbf{t}\in\closuretypes^{(s_{D-d})}_k$. The event $\closuretypeof_{s_{D-d}}(\overline{u})=\mathbf{t}$ defines the $s_{D-d}$-closure types $\mathbf{t}_i$ of $\overline{u}_i$, since
 $\closure_{s_{D-d}}(\overline{u})\supseteq\closure_{s_{D-d}}(\overline{u}_i)\supseteq\closure_{s_{D-d_i}}(\overline{u}_i)$ for every $i\in[m]$, due to Fact~\ref{fact:closure-inclusion}.

We assume that the event $\mathcal{L}\wedge\{\closuretypeof_{s_{D-d}}(\overline{u})=\mathbf{t}\}$ holds in what follows.
By induction, we get that $\varphi^{(i)}:=\varphi^{(i)}(\mathbf{t}_i)$, $i\in[m]$, have been already defined and that terms $\tau_1,\ldots,\tau_m$ satisfy the inductive assertion. In particular, there are constants $K_1,\ldots,K_m$ and $\varepsilon_1,\ldots,\varepsilon_m$ such that $\mathbb{P}(n^{-K_i}\leq \tau_i\leq n^{K_i}\mid \tau_i\neq 0)=1$ and, with probability $1-n^{-\omega(1)}$, $|\tau_i-\varphi_n^{(i)}|\leq n^{-\varepsilon_i}\cdot\varphi_n^{(i)}$, for every $i\in[m]$. We apply Theorem~\ref{cl:function-concentration} with 
$$
\xi_i:=\tau_i, \, \varphi_i:=\varphi^{(i)}_n, \, \delta:=n^{-\omega(1)},
$$
$$
 \delta':=n^{-\min\{\varepsilon_1,\ldots,\varepsilon_m\}}, \, \text{ and }\,
 K:=\max\{K_1,\ldots,K_m\}.
$$
We immediately get that, for some constants $a,b>0$ (that only depend on $f$), 
$$
\mathbb{P}(\tau\in[K^{-b},K^b]\mid\tau\neq 0)=1,
$$
and either $\varphi_n\equiv 0$ and $\mathbb{P}(\tau=0)=1-n^{-\omega(1)}$, or $\varphi_n\geq K^{-b}$ and 
$$
\mathbb{P}\left(|\tau-\varphi_n|>a\delta'\cdot\varphi_n\right)=1-n^{-\omega(1)}.
$$

Finally, we assume that $\tau \in \mathrm{Agg}_{\functionclasspoly}$ and, in particular, $f\in\functionclasspoly$. Let  $I=\{i_1,\ldots,i_{m'}\}\subset[m]$ be the set of  $i\in[m]$ such that $\varphi^i_n\not\equiv 0$. We also assume, by induction, that, for every $i\in I$, $\varphi_n^i=(c_i+o(1))n^{\gamma_i}$ for some positive $c_i$ and real $\gamma_i$. Since $f\in\functionclasspoly$, we get that $f_I$ is asymptotically polynomial. Then, either $f_I\equiv 0$ and $\varphi_n=0$, or the equality $\tau=(c+o(1))n^{\gamma}$ holds
for some $c>0$ and $\gamma\in\mathbb{R}$, completing the proof for the case of function application.\\

Due to Remark~\ref{rk:no-inf}, it remains to handle two kinds of aggregations. Fix a closure type $F_0[(1,\ldots,k)]=\mathbf{t}_0\in\closuretype^{(s_{D-d})}_k$. In what follows, we assume that the event $\closuretypeof_{s_{D-d}}(\overline{u})=\mathbf{t}_0$ holds and define the value of $\varphi_n(\mathbf{t}_0)$.  Let $\closuretypes'_{F_0}\subset\closuretypes_{k+1}^{(s_{D-d+1})}$ be the set of all $s_{D-d+1}$-closure types $F[(1,\ldots,k+1)]$ such that
\begin{itemize}
\item $F\upharpoonright [k]=F_0\upharpoonright [k]$, 
\item $\closure_{s_{D-d+1}}^{F\cup F_0}([k+1])=F$,
\item the pair $(F_0,F\cup F_0)$ is sparse. 
\end{itemize}
Recall that the set $\closuretypes'_{F_0}\subset \closuretypes_{k+1}^{(s_{D-d+1})}$ is finite.

Due to Fact~\ref{fact:pair-of-closures-sparse}, for every $v\notin\closure_{s_{D-d}}(u)$, we have that 
$\closure_{s_{D-d+1}}(\overline{u},v)\cup\closure_{s_{D-d}}(\overline{u})$
is a sparse extension of 
$\closure_{s_{D-d}}(\overline{u})$. 
So, there exists $\mathbf{t}'\in\closuretypes'_{F_0}$ such that $(\overline{u},v)$ has $s_{D-d+1}$-closure type $\mathbf{t}'$.
Therefore, the tuple $(\overline{u},v)$ has $\closuretypeof_{s_{D-d+1}}(\overline{u},v)\in\mathcal{C}'_{F_0}$. Let us show the opposite: that every closure type from $\closuretypes'_{F_0}$ is achievable by some $(\overline{u},v)$ with probability $1-n^{-\omega(1)}$. Fix $F[(1,\ldots,k+1)]\in\closuretypes'_{F_0}$. 
By Theorem ~\ref{thm:extensionexpectation}, with probability $1-n^{-\omega(1)}$ there exists an $(F_0,F\cup F_0)$-extension $F'_2\cup F'_1$ of $F'_1:=\closure_{s_{D-d}}(\overline{u})$ such that
\begin{itemize}
\item $F'_2$ is a $([k],F)$-extension of $\overline{u}$. Recall from the last bullet item of Definition \ref{def:extensions} that this means 
an extension that agrees with $F$ on all edges outside of~$[k]$. 
\item $F'_2$ does not have $s_{D-d+1}$-dense extensions $F'_3$ in $\mathbf{G}_n$ with $V(F'_3)\setminus (V(F'_1)\cup V(F'_2))\neq\varnothing$ and $E(F'_3)\cap((V(F'_2)\setminus V(F'_1))\times(V(F'_3)\setminus (V(F'_1)\cup V(F'_2)))\neq\varnothing$.
\end{itemize}
Let $v$ be the image of $k+1$ under a corresponding isomorphism of extensions $F\to F'_2$. Let us show that $\closuretypeof_{s_{D-d+1}}(\overline{u},v)=F[(1,\ldots,k+1)]$. Note that $F'_3\setminus F'_2$ cannot be entirely inside $F'_1$ by the definition of $\closuretypes'_{F_0}$. It then suffices to prove that $F'_2$ does not have $s_{D-d+1}$-extensions that have some vertices outside of $F'_2\cup F'_1$. Assume the opposite --- let $F'_3$ be such an extension of $F'_2$. We also know that there are no edges between $V(F'_3)\setminus(V(F'_1\cup F'_2))$ and $V(F'_2)\setminus V(F'_1)$. Since the pair $(F'_3,F'_2)$ is $s_{D-d+1}$-dense, by Fact~\ref{fact:dense-sub-extension}, we get that $F'_1\cup(F'_3\setminus(F'_1\cup F'_2))$ is an $s_{D-d+1}$-dense extension of $F'_1$ --- condradiction with the fact that $F'_1=\closure_{s_{D-d}}(\overline{u})$.

\paragraph{Summation approximation.} $\tau:=\tau(\overline{u})=\sum_v\tau'(\overline{u},v)$, where, by the induction hypothesis, there exist constants $\varepsilon'>0,K'>0$ and a map $\varphi'_n:\closuretype_{k+1}\to\mathbb{R}_{\geq 0}$ such that, for every closure type $\mathbf{t}$ and every tuple $(\overline{u},v)$,
\begin{align*}
\mathbb{P}\left(|\tau'(\overline{u},v)-\varphi'_n(\mathbf{t})|>n^{-\varepsilon'}\varphi'_n(\mathbf{t})\mid \closuretypeof_{s_{D-d+1}}(\overline{u},v)=\mathbf{t}\right)
=n^{-\omega(1)}
\end{align*}
and $\mathbb{P}(n^{K'}\leq\tau'(\overline{u},v)\leq n^{K'}\mid \tau'(\overline{u},v)\neq 0)=1$. Moreover, either $\varphi'_n(\mathbf{t})\equiv 0$ or $\varphi'_n(\mathbf{t})\geq n^{-K'}$.

We immediately get that $\tau(\overline{u})\leq n^{K'+1}$ a.s. Moreover, if $\tau(\overline{u})\neq 0$, then there exists $v$ such that $\tau'(\overline{u},v)\neq 0$ which, in turn, implies $\tau\geq\tau'(\overline{u},v)\geq n^{-K'}$, by the induction assumption. Therefore,
$ \mathbb{P}(\tau\geq n^{-K'}\mid\tau\neq 0)=1$.

We then fix $v\in[n]$ and show that $\mathbf{t}_0$ defines the type of $(\overline{u},v)$, in each of the following three scenarios.

\begin{itemize}
\item $v\in\overline{u}$ (say $v=u_i$). In this case, 
$$
\closure_{s_{D-d+1}}(\overline{u},v)=\closure_{s_{D-d+1}}(\overline{u})\subseteq \closure_{s_{D-d}}(\overline{u}).
$$
Therefore, $\mathbf{t}_0$ identifies the $s_{D-d+1}$-closure type $\mathbf{t}_i$ of $(\overline{u},v)$.
\item $v\in\closure_{s_{D-d}}(\overline{u})\setminus\overline{u}$. Then, since  $k+1\leq W$, we get that 
$$
s_{D-d}\geq\ell(W,s_{D-d+1})\geq\ell(k+1,s_{D-d+1}).
$$
Since the event $\mathcal{L}$ holds, we can assume that $\closure_{s_{D-d+1}}(\overline{u},v)$ has at most $s_{D-d}$ vertices. By Fact ~\ref{fact:closure-intersection}, we get that 
$$
\closure_{s_{D-d+1}}(\overline{u},v)\subseteq\closure_{s_{D-d}}(\overline{u}).
$$
Therefore, $\mathbf{t}_0$ identifies the $s_{D-d+1}$-closure type of $(\overline{u},v)$. We denote it by $\mathbf{t}(v')$ for the image $v'$ of $v$ under the isomorphism of MRGs $\mathrm{cl}_{s_{D-d}}(\overline{u})[\overline{u}]$ and $F_0[(1,\ldots,k)]$.
\item $v\notin\closure_{s_{D-d}}(\overline{u})$. We know that there exists $\mathbf{t}'\in\closuretypes'_{F_0}$ such that $(\overline{u},v)$ has $s_{D-d+1}$-closure type $\mathbf{t}'$.
\end{itemize}

\begin{definition}
For $F[(1,\ldots,k+1)]\in\mathcal{C}'_{F_0}$, let $\eta_F$ be the number of  $x\in V(F)$ such that there exists an automorphism of MRG $(F\cup F_0)[V(F_0)]$ that maps $k+1$ to $x$.
\end{definition}
Define 
\begin{equation}
 \label{eq:sparse_sum_pgi_definition-app}
 \varphi_n:=\sum_{i\in[k]}\varphi'_n(\mathbf{t}_i)+\sum_{v'\in F_0\setminus R_0}\varphi'_n(\mathbf{t}(v'))+
 \sum_{F[R]=\mathbf{t}\in\mathcal{C}'_{F_0}}\varphi'_n(\mathbf{t})\cdot\mu(F_0,F\cup F_0;s_{D-d+1})\cdot\eta_F.
\end{equation}
For brevity, we denote 
$$
 \mu_F:=\mu(F_0,F\cup F_0;s_{D-d+1})\cdot\eta_F.
$$
In what follows, we show that, with probability $1-n^{-\omega(1)}$, for every $F[(1,\ldots,k+1)]=\mathbf{t}\in\mathcal{C}'_{F_0}$, the number of vertices $v$ such that $(\overline{u},v)$ has $s_{D-d+1}$-closure type $\mathbf{t}$ equals $(1\pm n^{-\varepsilon''})\mu_F$, for some $\varepsilon''>0$. This would complete the proof of concentration in the $\functionclassrellip$-case. Indeed, if all summands in~\eqref{eq:sparse_sum_pgi_definition-app} equal 0, then $\varphi_n=0$ and, by induction assumption and the union bound, with probability $1-n^{-\omega(1)}$ for every $v$, $\tau'(\overline{u},v)=0$, implying $\tau(\overline{u})=0$. If at least one summand $\varphi'_n(\mathbf{t})$ in~\eqref{eq:sparse_sum_pgi_definition-app} does not equal 0, then $\varphi_n\geq\varphi'_n(\mathbf{t})\geq n^{-K'}$. Moreover, 
\begin{align*}
 |\tau-\varphi_n|&\leq\sum_{v\in V(F'_1)}|\tau(\overline{u},v)-\varphi'_n(\closuretypeof_{s_{D-d+1}}(\overline{u},v))|+\left|\sum_{v\notin V(F'_1)}\tau(\overline{u},v)-\sum_{\mathbf{t}\in\mathcal{C}'_{F_0}}\varphi'_n(\mathbf{t})\cdot\mu_F\right|\\
 &\leq n^{-\varepsilon'}\sum_{v\in V(F'_1)}\varphi'_n(\closuretypeof_{s_{D-d+1}}(\overline{u},v))+\sum_{\mathbf{t}\in\mathcal{C}'_{F_0}}\left|(1\pm n^{-\varepsilon''}\pm n^{-\varepsilon'})\mu_F\varphi'_n(\mathbf{t})-\varphi'_n(\mathbf{t})\mu_F\right|\\
 &\leq (n^{-\varepsilon'}+n^{-\varepsilon''})\varphi_n,
\end{align*}
as needed.

So, in order to complete the proof of concentration, it remains to prove that, for every $\mathbf{t}=F[(1,\ldots,k+1)]\in\mathcal{C}'_{F_0}$, $\mu_F$ approximates well the number of vertices $v$ such that $\closuretypeof_{s_{D-d+1}}(\overline{u},v)=\mathbf{t}$. Due to Theorem~\ref{thm:extensionexpectation}, the number of $(F_0,F\cup F_0)$-extensions $F':=F'_2\cup F'_1$ of $F'_1:=\closure_{s_{D-d}}(\overline{u})$ such that $F'_2$ is a $([k],F)$-extension of $\overline{u}$ and $F'$ does not have $s_{D-d+1}$-dense extensions equals $(1\pm n^{-\varepsilon''})\mu(F_0,F\cup F_0;s_{D-d+1})$. On the one hand, every $v\notin V(F'_1)$ has $\mathrm{cl}_{s_{D-d+1}}(\overline{u},v)=:F'_2$ such that $F'_2[(\overline{u},v)]\cong F[(1,\ldots,k+1)]$ for some $F[(1,\ldots,k+1)]\in\mathcal{C}'_{F_0}$. On the other hand, for every $F[(1,\ldots,k+1)]\in\mathcal{C}'_{F_0}$ and every $(F_0,F\cup F_0)$-extension $F'$ of $F'_1$, the number of vertices $v\in V(F')\setminus V(F'_1)$ such that $\mathrm{cl}_{s_{D-d}}(\overline{u})\cup\mathrm{cl}_{s_{D-d+1}}(\overline{u},v)=F'$ is exactly $\eta_F$, as required.

Finally, let $\tau \in \mathrm{Agg}_{\functionclasspoly}$. We assume that, for every $v\in[n]$ and for $\mathbf{t}:=\closuretypeof_{s_{D-d+1}}(\overline{u},v)$, 
$$
\varphi'_n(\mathbf{t})=(c'(\mathbf{t})+o(1))\cdot n^{\gamma'(\mathbf{t})}.
$$
The case of $\varphi'_n(\closuretypeof_{s_{D-d+1}}(\overline{u},v))\equiv 0$ for all $v\in[n]$ is discussed above. Therefore, we assume $\varphi'_n(\closuretypeof_{s_{D-d+1}}(\overline{u},v))\not\equiv 0$,  for some $v\in[n]$. For $F[(1,\ldots,k+1)]\in\mathcal{C}'_{F_0}$, we let $\beta_F:=\lim_{n\to\infty}\ln\mu_F/\ln n$.
Note that  Theorem~\ref{thm:extensionexpectation} states that there are constants $a,b$ such that $\mu_F=(a+o(1))n^b$. Therefore, $\ln\mu_F=(b+o(1))\ln n$.
Thus the limit exists.
 From Definition~\eqref{eq:sparse_sum_pgi_definition-app} and ~\eqref{eq:exetnsions-expectation-asymp}, we get that 
$$
\varphi_n=(c+o(1))\cdot n^{\gamma},
$$
where  
$$
 \gamma=\left\{\max_{v\in V(F'_1)}\gamma'(\closuretypeof_{s_{D-d+1}}(\overline{u},v)),\,\max_{F[R]=\mathbf{t}\in\mathcal{C}'_{F_0}}\beta_F\gamma'(\mathbf{t})\right\},
$$
and
\begin{align*}
 c &=\sum_{v\in V(F'_1):\,\gamma'(c)=\gamma}c'(\closuretypeof_{s_{D-d+1}}(\overline{u},v))+\sum_{F[(1,\ldots,k+1)]=\mathbf{t}\in\mathcal{C}'_{F_0}:\,\gamma'(\mathbf{t})=\gamma}\frac{c'(\mathbf{t})\eta_F}{\mathrm{aut}(F_0,F\cup F_0)},
\end{align*}
completing the proof for the case of summation approximation.

\paragraph{$\supl$ approximation.} $\tau(\overline{u})=\max_v\tau'(\overline{u},v)$. By the induction hypothesis, there exist constants $\varepsilon',K'>0$ and a map $\varphi'_n:\closuretype_{k+1}\to\mathbb{R}_{\geq 0}$ as in the previous case. We immediately get that $\tau(\overline{u})\leq n^{K'}$ a.s. and that $\mathbb{P}(\tau(\overline{u})\geq n^{-K'}\mid\tau\neq 0)=1$. Indeed, if $\tau(\overline{u})\neq 0$, then there exists $v$ such that $\tau'(\overline{u},v)\neq 0$ which, in turn, implies $\tau(\overline{u})\geq\tau'(\overline{u},v)\geq n^{-K'}$. 

As in the previous case of summation agregation, we fix $v\in[n]$ and observe that $\mathbf{t}_0$ defines the type of $(\overline{u},v)$, in each of the three cases 1) $v\in\overline{u}$, 2) $v\in\mathrm{cl}_{s_{D-d}}(\overline{u})\setminus\overline{u}$, and 3) $v\notin\mathrm{cl}_{s_{D-d}}(\overline{u})$. Define  $\varphi_n$ to be the maximum of
\[
\max_{i\in[k]}\varphi'_n(\mathbf{t}_i), \,
 \max_{v\in\closure_{s_{D-d}}(\overline{u})\setminus\overline{u}}\varphi'_n(\mathbf{t}(v)),  \,
\max_{\mathbf{t}'\in\closuretypes'_{F_0}}\varphi'_n(\mathbf{t}')
 \]
We have that $\varphi_n\equiv 0$ if and only if, for every $v\in[n]$,  $\varphi'_n(\closuretypeof_{s_{D-d+1}}(\overline{u},v))\equiv 0$.
 
Since all the closure types from $\closuretypes'_{F_0}$ are achievable by $(\overline{u},v)$, then $\varphi_n=\varphi'_n(\closuretypeof_{s_{D-d+1}}(\overline{u},v))$ for some $v$. Therefore, on the one hand, with probability $1-n^{-\omega(1)}$,
\begin{align*}
\tau(\overline{u})\geq\tau'(\overline{u},v)\geq \varphi'_n(\closuretypeof_{s_{D-d+1}}(\overline{u},v))(1-n^{-\varepsilon'})=\varphi_n(1-n^{-\varepsilon'}).
\end{align*}
On the other hand, with probability $1-n^{-\omega(1)}$,
\begin{align*}
 \tau(\overline{u})\leq\max_v\varphi'_n(\closuretypeof_{s_{D-d+1}}(\overline{u},v))(1+n^{\varepsilon'})=\varphi_n(1+n^{-\varepsilon'}),
\end{align*}
completing the proof of concentration.

Finally, if $\tau \in \mathrm{Agg}_{\functionclasspoly}$, then, applying the induction hypothesis and the definition of $\varphi_n$, we immediately get that $\varphi_n=(c+o(1))n^{\gamma}$ for some $c>0$ and $\gamma\in\mathbb{R}$, completing the proof.

\end{proof}

\subsection{Convergence law for average and $\max$}
\label{subsc:convergence-law-sparse}

Recall from Section \ref{sec:related} that \cite{usneurips24} proves convergence for the sparse case, but with \emph{only} averages, not maximum or minimum. Using the techniques developed for our concentration results, we can show that adding $\supl$ to averaging operators we still have the convergence law.

\begin{theorem} \label{thm:aasshalehspencerirrationalavg}
For every closed $\tau\in\Aggo{\Lipsch, \Mean, \LMean, \supl}$, its evaluation $[\![\tau]\!]_{\mathbf{G}_n}$  
converges in probability to a constant.
\end{theorem}

Since the proof of Theorem~\ref{thm:aasshalehspencerirrationalavg} follows the same ideas as the proof of Theorem~\ref{th:sparse-concentration-and-limits} and is in fact simpler, we defer it to Appendix~\ref{app:irrational-proof}.

Although the convergence result here is fairly strong, it has the following weakness: averaging loses information. This is the perspective  in \cite{usneurips24,adamday2023zeroone}, where the convergence phenomena is related to ``oversmoothing'' in graph neural networks.  Seen from the point of view of understanding asymptotic combinatorics of graphs,  we find that many important  properties of sparse random graphs with well-behaved asymptotics do not converge to a constant, and hence cannot be expressed in the term language of Theorem~\ref{thm:aasshalehspencerirrationalavg}.
Consider, for instance, the number of triangles in a sparse random graph, say with $\alpha>\frac{2}{3}$. A typical vertex does not belong to any triangle. So, the average number of vertices that belong to a triangle converges to 0. 
 Nevertheless, as in the dense case, when $\alpha<1$, the number of triangles
in $\mathbf{G}_n$ is concentrated around $n^{\Theta(1)}$. 
  In this sense, Theorem~\ref{th:sparse-concentration-and-limits} gives a stronger result that captures important graph characteristics, as in Corollary~\ref{cor:sparse-cor}.

\section{Comparison with previous results} \label{sc:comparison}

We claimed in Section~\ref{sec:related} that our results imply some prior convergence results. We explain in detail here.

\subsection{Zero-one laws for FO logic}

We show that our results generalise the classical zero-one law for first-order logic~\cite{faginzeroone,glebskii}. First, for every sentence in first order logic, its Boolean connectives can be replaced with functions from $\functionclasspoly$. Indeed, let $f:\{0,1\}^m\to\{0,1\}$ be a Boolean function. Its extension $f^*\in\functionclasspoly$ can be defineed as folows: $f^*(\mathbf{0})=f(\mathbf{0})$, where $\mathbf{0}=(0,\ldots,0)\in\{0,1\}^m$. Then, for every $x\in\{0,1\}^m\setminus\{\mathbf{0}\}$ such that $f(x)=0$, we let $I(x)$ be the set of $i\in[m]$ such that $x_i\neq 0$ and define $f^*_{I(x)}\equiv 0$. For all the other non-empty $I\subseteq[m]$, we let $f^*_I\equiv 1$. 
  Next, every universal and existential quantifier can be replaced with aggregators $\supl, \infl$, respectively. Since every such term has value 0 or 1, Theorem~\ref{th:dense-limits} immediately implies that it converges in probability either to 0 to 1, which is exactly the same as convergence to 0 or 1 of the probability of validity of the respective sentence. Similarly, Theorem~\ref{th:sparse-concentration-and-limits} implies the zero-one law of Shelah and Spencer for the irrational sparse case. 

\subsection{Convergence laws for real-valued logics}

Theorem~\ref{th:dense-limits} also generalises the convergence law for real-valued logic~\cite{uslics}, when the connectors are relative Lipschitz and asymptotically polynomial. In particular, due to Corollary~\ref{cor:classes}, $f(x)=\frac{1}{x}\mathbbm{1}_{x>0}\in \functionclassrellip$ and $g(x_1,x_2)=x_1x_2\in \functionclassrellip$. Therefore, both functions also belong to $\functionclasspoly$. We conclude that the global average can be expressed as 
$$
 \frac{1}{n}\sum_v\tau(\cdot,v)=\sum_v g\left(f\left(\sum_u 1\right),\tau(\cdot,v)\right)
$$ 
and the local average can be expressed as 
$$
\frac{1}{n}\sum_{\mathrm{E}(u_i,v)}\tau(\overline{u},v)=
\sum_v g\left(f\left(\sum_v \mathbbm{1}_{\mathrm{E}(u_i,v)}\right),\tau(\overline{u},v)\right),
$$
where $\overline{u}=(u_1,\ldots,u_k)$.

Since \cite{uslics} deals with Lipschitz connectives, one may wonder whether our concentration results can be extended to Lipschitz functions.
We show that the answer is no, and in the process argue that the relative Lipschitz property is necessary.

Recall that the concentration result does not hold for bounded terms in the language considered in~\cite{uslics} extended by the aggregator $\Sigma$ --- for example it fails for $\mathrm{Agg}_{\{\sin x\}}$. In this example the connective function that we added was not convergent,
and it may look like this is the  only possible reason for getting non-convergence for bounded terms.
 However, as we show below, even when one deals with Lipschitz functions that are convergent, we are not guaranteed to get
concentation.

Define $\mathcal{F}_0$ as the set of all functions $f:\mathbb{R}^m_{\geq 0}\to\mathbb{R}_{\geq 0}$ that are both Lipschitz and relative Lipschitz and such that each specification $f_I$ is constant on $\mathbb{R}^{|I|}_{\geq 0}\setminus (0,C)^{|I|}$ for some $C\geq 0$.
Even when all connectors belong to $\mathcal{F}_0$, the requirement of being relative Lipschitz is essential to have concentration. More formally, there exists a Lipschitz function $f:\mathbb{R}_{\geq 0}\to\mathbb{R}_{\geq 0}$ such that $f$ is constant on $[1,\infty)$, but the conclusion of Theorem~\ref{th:dense-concentration} does not hold for $\mathrm{Agg}_{\mathcal{F}_0\cup\{f\}}$ --- see Claim~\ref{cl:counterexamples}.

On the other hand, there are functions that are not Lipschitz (even not continuous) that belong to $\functionclasspoly$ since specifications of $f$ should not be consistent. Therefore, our result in the dense case still applies to show convergence of some terms that are not covered by the convergence law of~\cite{uslics}.  

Finally, Theorem~\ref{thm:aasshalehspencerirrationalavg} extends the results from~\cite{usneurips24,uslics} to $\Aggo{\Lipsch,\LMean,\Mean,\supl}$ in the sparse irrational regime.

\section{Discussion} \label{sec:discussion}

In this paper we are concerned with two main cases of the growth rate of $p$ in the context of logical limit laws for  {\erdosrenyi} random graphs: $p=\mathrm{const}$ and $p=n^{-\alpha}$. We stress that the case of $p$ decaying more  slowly than any power of $n$ is exactly the same as the case of $p=\mathrm{const}$, as it happens with the FO 0-1 law --- see~\cite{spencer-dense} --- and we prove Theorems~\ref{th:dense-concentration},~\ref{th:dense-limits} only for constant $p$ just for the sake of clarity of presentation. Also, the case $p=n^{-\alpha}$ can be easily generalised to $p=\Theta(n^{-\alpha})$. As we note in the introduction, when $\alpha\in(0,1)$ is rational, the random graph does not satisfy the FO convergence law and, therefore, the conclusion of Theorem~\ref{th:sparse-concentration-and-limits} does not hold. The case where $\alpha>1$ is less interesting, since in this regime the random graph is a forest $\whp$ where the number of connected components of each isomorphism type is well understood. So the only missing case that would be interesting to investigate in the future in the context of asymptotic behaviour of terms is the case of linear sparse $G(n,p=c/n)$. We note that an analogue of Theorems~\ref{th:dense-concentration},~\ref{th:dense-limits},~\ref{th:sparse-concentration-and-limits} does not hold in this regime since the FO 0-1 fails. Nevertheless, terms in $\Aggo{\Lipsch,\LMean,\Mean,\supl}$ converge in distribution~\cite{uslics}.

The main results of this paper, both in the dense and sparse regimes, can be extended to relational structures with arbitrary signatures $\sigma$. In particular, let $\sigma=\{=,R_1,\ldots,R_k\}$, where $R_i$ has arity $i$, and let $G(n,p_1,\ldots,p_k)$ be the product measure on $\prod_{i=1}^k{[n]\choose i}$, where each $i$-subset is present with probability $p_i$, independently of the others. Then the same proof techniques can be applied to show the analogues of Theorems~\ref{th:dense-concentration},~\ref{th:dense-limits},~\ref{th:sparse-concentration-and-limits}  for this distribution with each $p_i$ being a constant or $n^{-\alpha_i}$ for some irrational $\alpha_i>0$ and the language is extended by adding relations from $\sigma$. It would be interesting to address other distributions on $\prod_{i=1}^k{[n]\choose i}$, as well as weighted random graph models. In particular, papers~\cite{usneurips24,uslics} also study graphs with independent and indentically distributed random weights (features) on its vertices which naturally arise in the context of graph neural networks. Clearly, our results can be extended to the weighted $G(n,p)$ when weights have a finite support $X\subset[0,\infty)$ and their distribution do not depend on $n$. It would be interesting and natural to generalise our results to other distributions on weights --- in particular, when the distribution depends on $n$.

The error terms in our main results can be optimised using the same proof methods, but getting optimal constant factors in the second order terms in asymptotical expansions of the terms in our languages seems a challenging task. In particular, we suspect that by excluding maximum aggregation and restricting the class of connective functions, it might be possible to prove a central limit theorem for $\frac{\tau-\mathbb{E}\tau}{\sqrt{\mathrm{Var}(\tau)}}$. It would be of significant interest to prove a more general result of this kind for the entire language $\Agglang_{\functionclass}$, a result that involves a wider class of limit distributions --- in particular, the Gumbel distribution, which is the limit for a rescaled maximum degree in the random graph~\cite{Bollobas_degrees,degreesequences,ivchenko}.

There are well-studied graph functions that are concentrated in $G(n,p)$ but are not covered in this paper. This is the case for terms in higher-order languages. In particular, it is known that the chromatic number and the independence number are well-concentrated~\cite{Bol-chrom,BE-independence,GM-independence,Matula1-independence,ShamirSpencer}. Nevertheless, the convergence law fails even for significantly small fragments of monadic second order --- in particular, this is the case for existential monadic second-order logic with two first order variables~\cite{AZ,PZ} (even though graph properties involving chromatic number and independence number are expressible in this logic). Although Theorems~\ref{th:dense-concentration},~\ref{th:dense-limits} cannot be generalised to a real-valued analogue of this fragment of monadic second-order logic, it remains an open question whether suitable additional restrictions on the term language could ensure concentration and subsume both the chromatic and independence numbers.

There are also interesting graph parameters that concentrate and that {\it can} be expressed in the ``first-order'' $\Agglang_{\functionclass}$ for some $\functionclass\neq\functionclassrellip$. In particular, if $\functionclass$ includes Boolean connectives and $\mathbbm{1}_{x=y}$, then it is possible to express the total number of vertices $X^{\max}$ that have maximum degree $\Delta$. This quantity equals 1 $\whp$ in $G(n,p)$ (and concentrates)~\cite{degreesequences}. Nevertheless, it is clear that it is not possible to express this term in $\Agglang_{\functionclassrellip}$. Indeed, let $\mathbf{G}^1\sim G(n,1/2)$ and let $\mathbf{G}^2$ be generated in the following way: Generate $G(n,1/2)$ and keep it unchanged with probability $1/2$; otherwise, with probability $1/2$, find the maximum degree $\Delta$, then take a uniformly random vertex, delete all edges it belongs to, and draw instead $\Delta$ uniformly random edges to all vertices excluding the one with the maximum degree. It is possible to show that $\mathbf{G}^1$ and $\mathbf{G}^2$ are industinguishable $\whp$, i.e., 
$$
\biggl|[\![\tau]\!]_{\mathbf{G}^1}-[\![\tau]\!]_{\mathbf{G}^2}\biggr|=o\biggl(\min\bigl\{[\![\tau]\!]_{\mathbf{G}^1},[\![\tau]\!]_{\mathbf{G}^2}\bigr\}\biggr)
$$ 
$\whp$, for every non-zero $\tau\in\Agglang_{\functionclassrellip}$. Nevertheless, $\whp$ $\mathbf{G}^1$ has one vertex with maximum degree, while $\mathbf{G}^2$ has two vertices with maximum degree, contradicting $X^{\max}\in\Agglang_{\functionclassrellip}$.

On the one hand, the discussion above motivates the following question. Whether the language $\Agglang_{\mathrm{Bool}\cup\{\mathbbm{1}_{x=y}\}}$, where $\mathrm{Bool}$ is the set of all Boolean connectives, obeys the concentration law;
and, if the answer is positive, can
one can define a fairly broad class $\mathcal{F}$ that includes Boolean connectives and $\mathbbm{1}_{x=y}$ preserving the concentration law? On the other hand, the above example shows that we can use our concentration results to prove inexpressiveness results for real-valued logics.

Finally, our results opens up several new computational challenges associated with the behaviour of terms in random graphs. For example, assuming a restricted language, so that all components involved in the definition of a term can be effectively defined, how hard is it to determine whether the term converges and find or approximate the limit when it does? We note that computational aspects of asymptotic behaviour of FO sentences evaluated on {\it dense} random graphs have been thoroughly studied in the literature, see,~e.g.,~\cite{Demin,GHK,Liogonkii-comp}.

\bibliographystyle{plain}
\bibliography{references}

\appendix

\section{A class of relative Lipschitz differentiable functions}
\label{app:rellip}

Here we state and prove a claim that shows that every differentiable function $f$ with $(\ln f(x))'=O((\ln x)')$ is relative Lipschitz. In particular, the sigmoid function is relative Lipschitz.

\begin{claim}
\label{cl:bounded_functions}
Every differentiable $f:\mathbb{R}_{>0}\to\mathbb{R}_{>0}$ such that $xf'(x)<Cf(x)$ for some $C>0$ and all $x>0$ is relative Lipschitz.
\end{claim}

\begin{proof}
 
Let $y=x+\delta>x>0$. It suffices to prove that $(2x+\delta)(f(x+\delta)-f(x))<(2C+1)\delta(f(x+\delta)+f(x))$. We note that the left hand side and the right hand size of this inequality are equal when $\delta=0$. For a fixed $x>0$, define
$$
 g(\delta)=(2x+\delta)(f(x+\delta)-f(x))-(2C+1)\delta(f(x+\delta)+f(x)).
$$
We get 
\begin{align*}
g' &=f(x+\delta)-f(x)+(2x+\delta)f'(x+\delta)-(2C+1)(f(x+\delta)+f(x))-(2C+1)\delta f'(x+\delta)\\&<2(x+\delta)f'(x+\delta)-2Cf(x+\delta)<0
\end{align*}
for all $\delta>0$ and all $x>0$, completing the proof.

\end{proof}

\section{Bounded non-coverging relative Lipschitz function}
\label{app:no-limit}

\begin{claim}
$2+\sin(\ln(2+x))$ is bounded relative Lipschitz function such that $\lim_{x\to\infty}f(x)$ does not exist.
\end{claim}

\begin{proof}
Non-existence of a limit and boundedness are obvious, so we only need to prove that $f\in\functionclassrellip$. Let $y>x>0$. Let $\varepsilon>0$ be small enough. If $y>(1+\varepsilon)x$, then 
$$
 \frac{y-x}{y+x}=1-\frac{2x}{y+x}>1-\frac{2y/(1+\varepsilon)}{y+y/(1+\varepsilon)}=\frac{\varepsilon}{2+\varepsilon},\quad
\text{ while }\quad
\frac{f(y)-f(x)}{f(y)+f(x)}<\frac{2}{f(y)+f(x)}\leq \frac{1}{2},
$$
implying
$$
 \frac{f(y)-f(x)}{f(y)+f(x)}<\frac{2+\varepsilon}{2\varepsilon}\cdot\frac{y-x}{y+x}.
$$

Finally, let $(1+\varepsilon)x\geq y>x$. Then 
$$
\sin\left(\frac{\ln(y/x)}{2}\right)=
\sin\left(\frac{\ln(1+(y-x)/x)}{2}\right)\sim_{\varepsilon\to 0}
\sin\frac{y-x}{2x}\sim_{\varepsilon\to 0}\frac{y-x}{2x}.
$$
Therefore,
$$
 \frac{f(y)-f(x)}{f(y)+f(x)}=
 \frac{\sin(\ln(y/x)/2)\cos(\ln(yx)/2)}{4+\sin(\ln(yx)/2)\cos(\ln(y/x)/2)}<
 \frac{(y-x)/(2x)}{2}<\frac{y-x}{y+x},
$$
completing the proof.
\end{proof}

\section{Proof of Theorem~\ref{thm:aasshalehspencerirrationalavg}} 
\label{app:irrational-proof}

Recall the statement of Theorem~\ref{thm:aasshalehspencerirrationalavg}:

\medskip

For every closed $\tau\in\Aggo{\Lipsch, \Mean, \LMean, \supl}$, its evaluation $[\![\tau]\!]_{\mathbf{G}_n}$  
converges in probability to a constant.

\medskip

As with our other theorems, we need an inductive invariant for open terms:

\begin{lemma}[Inductive Term Approximation: sparse case] \label{lem:sparse_avg_inductive}
Let $\tau \in  \Aggo{\Lipsch, \Mean, \LMean, \supl}$ be a term of depth $d\leq D$ with $k$ free variables and at most $W$ variables in total and let $s_0\geq s_1\geq\ldots\geq s_D=0$ be a $(D, W)$-rapid sequence. There exists constant $\varepsilon>0$ 
 and a map $\varphi:\closuretypes^{(s_{D-d})}_k\to \mathbb{R}_{\geq 0}$,  such that, for every $\overline{u}\in[n]^k$ and every $\mathbf{t}\in\closuretypes^{(s_{D-d})}_k$
$$
\mathbb{P}\left(|\tau(\overline{u})-\varphi(\mathbf{t})|>n^{-\varepsilon}\mid\mathcal{L}\wedge\{\closuretypeof(\overline{u})=\mathbf{t}\}\right)=n^{-\omega(1)}.
$$
\end{lemma}
In particular, for the equivalence class $C(\mathbf{t})$ of tuples with the same $s_{D-d}$-closure type $\mathbf{t}$, we get that both random variables $\max_{\overline{u}\in C(\mathbf{t})}\tau(\overline{u})$ and $\min_{\overline{u}\in C(\mathbf{t})}\tau(\overline{u})$ converge to $\varphi(\mathbf{t})$ in probability.

Theorem \ref{thm:aasshalehspencerirrationalavg} follows directly from Lemma~ \ref{lem:sparse_avg_inductive} by taking $k=0$.

\begin{proof}[Proof of Lemma~\ref{lem:sparse_avg_inductive}.]

We start from the atomic cases.

\paragraph{Atomic cases.}

\begin{enumerate}

\item $\tau\equiv\mathrm{const}$. In this case, the assertion is obvious for $\varphi=\tau$.

\item $\tau(u_1,u_2)=\mathbbm{1}_{\mathrm{E}(u_1,u_2)}$  is a term with two free variables. For $\mathbf{t}=F[(1,2)]\in\closuretypes_2^{s_D}$, we let $\varphi(\mathbf{t})=\mathbbm{1}_{E(F\upharpoonright [2])\neq\varnothing}$. Then $\tau(u_1,u_2)=\varphi(\mathbf{t})$ for every pair of vertices $(u_1,u_2)$ with $s_D$-closure type $\mathbf{t}$. The inductive assertion clearly follows. 

\item The case of $\tau(u_1,u_2)=\mathbbm{1}_{u_1=u_2}$ is analogous to the previous edge case.

\end{enumerate}

\paragraph{Function application.} $\tau=f(\tau_1,\ldots,\tau_m)$ is a term of depth $d$, where $f$ is a Lipshitz function. We have that each term $\tau_i$ has depth $d_i\leq d$. Let $\tau$ and $\tau_i$ have $k$ and $k_i$ free variables respectively. Let $\overline{u}\in[n]^k$. Substituting it as an input of $\tau$, identifies arguments of $\tau_1,\ldots,\tau_m$ --- we denote the respective subtuples of $\overline{u}$ by $\overline{u}_i$. For the sake of convenience and clarity of presentation, 
we assume that $\overline{u}\in([n])_k$. The case when $u_i=u_j$ for some $1\leq i<j\leq k$ is verbatim.

Fix an $s_{D-d}$-closure type of $k$-tuples $\mathbf{t}\in\closuretypes^{(s_{D-d})}_k$. The event $\closuretypeof_{s_{D-d}}(\overline{u})=\mathbf{t}$ defines the $s_{D-d}$-closure types $\mathbf{t}_i$ of $\overline{u}_i$, 
since
 $\closure_{s_{D-d}}(\overline{u})\supseteq\closure_{s_{D-d}}(\overline{u}_i)\supseteq\closure_{s_{D-d_i}}(\overline{u}_i)$ for every $i\in[m]$, due to Fact~\ref{fact:closure-inclusion}.

By induction, we get that $\varphi^{(i)}:=\varphi^{(i)}(\mathbf{t}_i)$, $i\in[m]$, have been already defined and that terms $\tau_1,\ldots,\tau_m$ satisfy the inductive assertion. 
In particular, there are constants $\varepsilon_1,\ldots,\varepsilon_m$ such that, with probability $1-n^{-\omega(1)}$, $|\tau_i-\varphi^{(i)}|\leq n^{-\varepsilon_i}$. Define $\varphi=f(\varphi^{(1)},\ldots,\varphi^{(m)})$. Since $f$ is Lipschitz, for some constant $C>0$,
\begin{align*}
 |\tau-\varphi| &=|f(\tau_1,\ldots,\tau_m)-f(\varphi^{(1)},\ldots,\varphi^{(m)})|\\
 &\leq C\|(\tau_1-\varphi^{(1)},\ldots,\tau_m-\varphi^{(m)}\|
 \leq C\sqrt{m}\cdot \max_i n^{-\varepsilon_i}\leq n^{-\varepsilon}
\end{align*}
with probability $1-n^{-\omega(1)}$, where $\varepsilon=\min\{\varepsilon_1,\ldots,\varepsilon_m\}/2$. This completes the proof.

It remains to handle three kinds of aggregations. 

\paragraph{Global average.} $\tau(\overline{u})=\frac{1}{n}\sum_v\tau'(\overline{u},v)$ is a term of depth $d$, where, by the induction hypothesis, there exist constant $\varepsilon'>0$ and a map $\varphi':\closuretypes^{(s_{D-d+1})}_{k+1}\to\mathbb{R}_{\geq 0}$ such that, for every closure type $\mathbf{t}\in\closuretypes^{(s_{D-d+1})}_{k+1}$ and every tuple $(\overline{u},v)$, 
\begin{equation}
\label{eq:mean_inductive_assumption_prob}
\mathbb{P}\left(|\tau'(\overline{u},v)-\varphi'(\mathbf{t})|> n^{-\varepsilon'}\mid \closuretypeof_{s_{D-d+1}}(\overline{u},v)=\mathbf{t}\right)\\=n^{-\omega(1)}.
\end{equation}
Fix $\overline{u}\in([n])_k$ and closure type $\mathbf{t}_0\in\closuretypes_k^{(s_{D-d})}$. In what follows, we assume that the event $\closuretypeof_{s_{D-d}}(\overline{u})=\mathbf{t}_0$ holds.  

We show that the only vertices that give non-negligible contribution to the value of $\tau(\overline{u})$ are the $v\notin\closure_{s_{D-d}}(\overline{u})$ that have no neighbours in $\overline{u}$ and such that $\overline{u}\cup v$ does not have $s_{D-d+1}$-dense extensions outside of $\mathrm{cl}_{s_{D-d}}(\overline{u})$.

For 
$v\notin\closure_{s_{D-d}}(\overline{u})$, the pair $(\overline{u},\mathbf{G}_n[\overline{u}\cup v])$ is sparse by the definition of a closure. Therefore, by Theorem ~\ref{thm:extensionexpectation}, the number of $v\notin\closure_{s_{D-d}}(\overline{u})$ that have neighbours in $\overline{u}$ is $O(np)=O(n^{1-\alpha})$, with probability $1-n^{-\omega(1)}$. 

Moreover, since the event $\mathcal{L}$ holds, we may assume that 
$$
|\closure_{s_{D-d+1}}(\overline{u}\cup v)|\leq s_{D-d}.
$$
Due to Fact~\ref{fact:pair-of-closures-sparse},
for every $v\notin\closure_{s_{D-d}}(\overline{u})$, the pair 
$$
(F'_1,F'_2):=(\closure_{s_{D-d}}(\overline{u}),\closure_{s_{D-d+1}}(\overline{u}\cup v)\cup \closure_{s_{D-d}}(\overline{u}))
$$
is sparse. 

In order to prove that vertices $v$ with non-empty set $V(F'_2)\setminus V(F'_1)$ do not contribute asymptotically to $\frac{1}{n}\sum\tau'(\overline{u},v)$, let us show that 
\begin{equation}
v(F'_1,F'_2)-\alpha\cdot e(F'_1,F'_2)<1.
\label{eq:v-e<1}
\end{equation}
Let $\mathbf{G}_n[\overline{u}\cup v]=H_0\subset H_1\subset\ldots\subset H_t=\closure_{s_{D-d+1}}(\overline{u}\cup v)$ be a sequence of $s_{D-d+1}$-dense extensions. By Fact~\ref{fact:dense-sub-extension}, pairs $(F'_1\cup H_{i-1},H_i)$ are $s_{D-d+1}$-dense as well. Every such dense pair contributes $v(F'_1\cup H_{i-1},H_i)-\alpha\cdot e(F'_1\cup H_{i-1},H_i)<0$ to the left hand side of~\eqref{eq:v-e<1} whenever the inclusion $F_1'\cup H_{i-1}\subset H_i$ is proper. Otherwise, the contribution is zero. Therefore,
$$
v(F'_1,F'_2)-\alpha e(F'_1,F'_2)
=
1-\alpha e(F'_1,\mathbf{G}_n\upharpoonright (V(F'_1)\cup v))
+\sum_{i=1}^t(v(F'_1\cup H_{i-1},H_i)-\alpha e(F'_1\cup H_{i-1},H_i))<1,
$$
as needed.

By Theorem ~\ref{thm:extensionexpectation}, for every sparse $(F_1,F_2)$ such that $V(F_1)=[V(\closure_{s_{D-d}}(\overline{u}))]$, $v(F_1,F_2)\leq s_{D-d}$, and $v(F_1,F_2)-\alpha\cdot e(F_1,F_2)<1$, the number of $(F_1,F_2)$-extensions of $\mathrm{cl}_{s_{D-d}}(\overline{u})$   is concentrated around 
$$
\mu(F_1,F_2;0)=\Theta\left(n^{v(F_1,F_2)-\alpha\cdot e(F_1,F_2)}\right)\stackrel{\eqref{eq:v-e<1}}<\frac{1}{2}n^{1-\varepsilon''}
$$ 
with probability $1-n^{-\omega(1)}$, for some constant $\varepsilon''>0$. 

We conclude that, 
by the union bound, with probability $1-n^{-\omega(1)}$, the number of $v\in[n]$ such that
\begin{itemize}
\item either $v\in \closure_{s_{D-d}}(\overline{u})$,
\item or $v\notin\closure_{s_{D-d}}(\overline{u})$ and the $s_{D-d+1}$-closure type $\mathbf{t}=F[(1,\ldots,k+1)]$ of $(\overline{u},v)$ has $|V(F)|>k+1$,
\item or $v\notin\closure_{s_{D-d}}(\overline{u})$ is adjacent to at least one vertex in~$\overline{u}$,
\end{itemize}
is at most $n^{1-\varepsilon''}$.

Let $\mathbf{t}^*\in\closuretypes_{k+1}^{(s_{D-d+1})}$ have $k+1$ vertices,
 where  $k+1$ is an isolated vertex in $\mathbf{t}^*$. 
By the inductive assumption, with probability $1-n^{-\omega(1)}$ for every $v\notin\overline{u}$ such that $\closuretypeof_{s_{D-d+1}}(\overline{u},v)=\mathbf{t}^*$, we have that $|\tau'(\overline{u},v)-\varphi'(\mathbf{t}^*)|\leq n^{-\varepsilon'}$.

Finally, we conclude that, with probability $1-n^{-\omega(1)}$, 
\begin{align*}
|\tau(\overline{u})-\varphi'(\mathbf{t}^*)|
&=
\left|\frac{1}{n}\left(\sum_{v:\,\closuretypeof_{s_{D-d+1}}(\overline{u},v)=\mathbf{t}^*}\tau(\overline{u})\right)-\varphi'(\mathbf{t}^*)+O(n^{-\varepsilon''})\right|\\
&\leq
\left|\frac{n-O(n^{1-\varepsilon''})}{n}\varphi'(t^*)-\varphi'(\mathbf{t}^*)+O(n^{-\varepsilon''})\right|+n^{-\varepsilon'}\\
&\leq n^{-\varepsilon'}+O(n^{-\varepsilon''}).
\end{align*}
Letting $\varphi(\mathbf{t})=\varphi'(\mathbf{t}^*)$ and $\varepsilon=\min\{\varepsilon',\varepsilon''\}/2$ completes the proof.

\paragraph{Local average.} $\tau(\overline{u})=\frac{1}{|N(u_1)|}\sum_{v:\,\mathrm{E}(v,u_1)}\tau'(\overline{u},v)$ is a term of depth $d$, where, as in the case of global average, there exist $\varepsilon'>0$ and $\varphi':\closuretypes^{(s_{D-d+1})}_{k+1}\to\mathbb{R}_{\geq 0}$ such that, for every closure type $\mathbf{t}\in\closuretypes^{(s_{R-d+1})}_{k+1}$ and every tuple $(\overline{u},v)$,~\eqref{eq:mean_inductive_assumption_prob} holds.
As above, we fix $\overline{u}\in([n])_k$ and assume that the event $\closuretypeof_{s_{D-d}}(\overline{u})=\mathbf{t}_0$ holds for some $\mathbf{t}_0\in\closuretypes_k^{(s_{D-d})}$.

Since here we normalise by a factor which is typically $n^{1-\alpha}$, we should take into account some atypical extensions as well, and not only the extensions by an isolated vertex  On the other hand, non-neighbours of $u_1$ do not contribute to the expression for $\tau$, and the closure types of $(\overline{u},v)$, where $v$ is adjacent to $u_1$, have $o(n^{1-\alpha})$ representatives when the closure of $(\overline{u},v)$ is non-trivial. More formally, using the same argument as in the case of global average, since non-trivial dense pairs contribute negatively to the extension count as in~\eqref{eq:v-e<1}, we get that,  for a sufficiently small contant $\varepsilon''>0$, due to Theorem ~\ref{thm:extensionexpectation} and by the union bound, with probability $1-n^{-\omega(1)}$, the number of neighbours $v\in[n]$ of $u_1$ such that 
\begin{itemize}
\item either $v\in \closure_{s_{D-d}}(\overline{u})$, 
\item or $v\notin\closure_{s_{D-d}}(\overline{u})$ and the $s_{D-d+1}$-closure type $\mathbf{t}$ of $(\overline{u},v)$ has $|V(\mathbf{t})|>k+1$, 
\item or $v\notin\closure_{s_{D-d}}(\overline{u})$ is adjacent to at least two vertices in $\overline{u}$,
\end{itemize}
is at most $n^{1-\alpha-\varepsilon''}$. 

Let $\mathbf{t}^*\in\closuretypes_{k+1}^{(s_{D-d+1})}$ have $k+1$ vertices, where the vertex $k+1$ is only adjacent to $1$.   By the inductive assumption, with probability $1-n^{-\omega(1)}$ for every $v\notin\overline{u}$ such that $\closuretypeof_{s_{D-d}}(\overline{u},v)=\mathbf{t}^*$, we have that $|\tau'(\overline{u},v)-\varphi'(\mathbf{t}^*)|\leq n^{-\varepsilon'}$. Moreover, by Theorem ~\ref{thm:extensionexpectation}, the vertex $u_1$ has $n^{1-\alpha}\pm n^{1-\alpha-\varepsilon'''}$ neighbours with probability $1-n^{-\omega(1)}$, for some $\varepsilon'''>0$. 

Let $\varphi(\mathbf{t})=\varphi'(\mathbf{t}^*)$ and $\varepsilon=\min\{\varepsilon',\varepsilon'',\varepsilon'''\}/2$. We get that, with probability $1-n^{-\omega(1)}$, 
\begin{align*}
|\tau(\overline{u})-\varphi(\mathbf{t})|=|\tau(\overline{u})-\varphi'(\mathbf{t}^*)|
\leq n^{-\varepsilon'}+O\left(n^{-\varepsilon''}+n^{-\varepsilon'''}\right)<n^{-\varepsilon},
\end{align*}
as needed.

\paragraph{$\supl$ elimination.} $\tau(\overline{u})=\max_v\tau'(\overline{u},v)$ is a term of depth $d$. By the induction hypothesis, there exist $\varepsilon'>0$ and $\varphi':\closuretypes^{(s_{D-d+1})}_{k+1}\to\mathbb{R}_{\geq 0}$ as in the previous two cases. Fix $\overline{u}\in([n])_k$ and assume that the event $\closuretypeof_{s_{D-d}}(\overline{u})=\mathbf{t}_0$ holds for some 
$$
 \mathbf{t}_0=:F_0[(1,\ldots,k)]\in\closuretypes_k^{(s_{D-d})}.
$$

\begin{definition}[Set of achievable types $\closuretypes'$]
Let $\closuretypes'\subset\closuretypes_{k+1}^{(s_{D-d+1})}$ be the set of all $s_{D-d+1}$-closure types $F[(1,\ldots,k+1)]$ such that 
\begin{itemize}
\item $F\upharpoonright [k]=F_0\upharpoonright [k]$ and $V(F_0)\cap V(F)=[k]$, 
\item $\closure_{s_{D-d+1}}^{F\cup F_0}([k+1])=F$,
\item the pair $(F_0,F\cup F_0)$ is sparse. 
\end{itemize}
\end{definition}

We then fix $v\in[n]$ and show that $\mathbf{t}_0$ determines the type of $(\overline{u},v)$, in each of the following three scenarios. 

\begin{itemize}
\item $v\in\overline{u}$ (say $v=u_i$). In this case, 
$$
\closure_{s_{D-d+1}}(\overline{u},v)=\closure_{s_{D-d+1}}(\overline{u})\subseteq \closure_{s_{D-d}}(\overline{u}).
$$
Therefore, $\mathbf{t}_0$ identifies the $s_{D-d+1}$-closure type $\mathbf{t}_i$ of $(\overline{u},v)$. By the induction hypothesis, we know that $\tau'(\overline{u},v)$ is $n^{-\varepsilon'}$-close to $\varphi'(\mathbf{t}_i)$.
\item $v\in\closure_{s_{D-d}}(\overline{u})\setminus\overline{u}$. Then, since  $k+1\leq W$, we get that 
$$
s_{D-d}\geq\ell(W,s_{D-d+1})\geq\ell(k+1,s_{D-d+1}).
$$
Since the event $\mathcal{L}$ holds, we can assume that $\closure_{s_{D-d+1}}(\overline{u},v)$ has at most $s_{D-d}$ vertices. By Fact ~\ref{fact:closure-intersection}, we get that 
$$
\closure_{s_{D-d+1}}(\overline{u},v)\subseteq\closure_{s_{D-d}}(\overline{u}).
$$
Therefore, $\mathbf{t}_0$ identifies the $s_{D-d+1}$-closure type of $(\overline{u},v)$. We denote it by $\mathbf{t}(v)$.  
 By the induction hypothesis, we know that $\tau'(\overline{u},v)$ is $n^{-\varepsilon'}$-close to $\varphi'(\mathbf{t}(v))$.
\item $v\notin\closure_{s_{D-d}}(\overline{u})$. As in the case of global average, we have that 
$$
F'_2:=\closure_{s_{D-d+1}}(\overline{u},v)\cup\closure_{s_{D-d}}(\overline{u})
$$ 
is a sparse extension of 
$$
F'_1:=\closure_{s_{D-d}}(\overline{u}).
$$
So, there exists $\mathbf{t}'\in\closuretypes'$ such that $(\overline{u},v)$ has $s_{D-d+1}$-closure type $\mathbf{t}'$. 
 By the inductive assumption,  $\tau'(\overline{u},v)$ is $n^{-\varepsilon'}$-close to $\varphi'(\mathbf{t}')$ with probability $1-n^{-\omega(1)}$. 
\end{itemize}
Recall that the set $\closuretypes'\subset \closuretypes_{k+1}^{(s_{D-d+1})}$ is finite. Define  $\varphi$ to be the maximum of
\[
\max_{i\in[k]}\varphi'(\mathbf{t}_i),
 \max_{v\in\closure_{s_{D-d}}(\overline{u})\setminus\overline{u}}\varphi'(\mathbf{t}(v)), 
\max_{\mathbf{t}'\in\closuretypes'}\varphi'(\mathbf{t}')
 \]
 where $\mathbf{t}_i$, $\mathbf{t}(v)$, and $\mathbf{t}'$ are defined in the bullet items above.
 If all closure types from $\closuretypes'$ are achievable by $(\overline{u},v)$, then $\varphi=\varphi'(\closuretypeof_{s_{D-d+1}}(\overline{u},v))$ for some $v$. In this case, on the one hand, with probability $1-n^{-\omega(1)}$,
$$
 \tau(\overline{u})\geq\tau'(\overline{u},v)\geq \varphi'(\closuretypeof_{s_{D-d+1}}(\overline{u},v))-n^{-\varepsilon'}=\varphi-n^{-\varepsilon'}.
$$
On the other hand, with probability $1-n^{-\omega(1)}$,
$$
 \tau(\overline{u})\leq\max_v(\varphi'(\closuretypeof_{s_{D-d+1}}(\overline{u},v))+n^{\varepsilon'})=\varphi+n^{-\varepsilon'}.
$$
Therefore, in order to complete the proof, we need to show that $\whp$ all closure types from $\closuretypes'$ are achievable by $(\overline{u},v)$ for $v\notin\closure_{s_{D-d}}(\overline{u})$.

Fix $F[(1,\ldots,k+1)]\in\closuretypes'$. 
By Theorem ~\ref{thm:extensionexpectation}, with probability $1-n^{-\omega(1)}$ there exists an $(F_0,F\cup F_0)$-extension $F'_2\cup F'_1$ of $F'_1:=\closure_{s_{D-d}}(\overline{u})$ such that
\begin{itemize}
\item $F'_2$ is a $([k],F)$-extension of $\overline{u}$. Recall from the last bullet item of Definition \ref{def:extensions} that this means 
an extension that agrees with $F$ on all edges outside of~$[k]$. 
\item $F'_2$ does not have $s_{D-d+1}$-dense extensions $F'_3$ in $\mathbf{G}_n$ with $V(F'_3)\setminus (V(F'_1)\cup V(F'_2))\neq\varnothing$ and $E(F'_3)\cap((V(F'_2)\setminus V(F'_1))\times(V(F'_3)\setminus (V(F'_1)\cup V(F'_2)))\neq\varnothing$.
\end{itemize}
Since  $F'_2$ is a $([k],F)$-extension of $\overline{u}$, there is an isomorphism of extensions $F\to F'_2$. Let $v$ be the image of $k+1$ under this isomorphism.  
 Let us show that $\closuretypeof_{s_{D-d+1}}(\overline{u},v)=F[(1,\ldots,k+1)]$. Note that $F'_3\setminus F'_2$ cannot be entirely inside $F'_1$ by the definition of $\closuretypes'$. It then suffices to prove that $F'_2$ does not have $s_{D-d+1}$-extensions that have some vertices outside of $F'_2\cup F'_1$. Assume the opposite --- let $F'_3$ be such an extension of $F'_2$. We also know that there are no edges between $V(F'_3)\setminus(V(F'_1\cup F'_2))$ and $V(F'_2)\setminus V(F'_1)$. Since the pair $(F'_3,F'_2)$ is $s_{D-d+1}$-dense, by Fact~\ref{fact:dense-sub-extension}, we get that $F'_1\cup(F'_3\setminus(F'_1\cup F'_2))$ is an $s_{D-d+1}$-dense extension of $F'_1$ --- condradicting the fact that $F'_1=\closure_{s_{D-d}}(\overline{u})$. This completes the proof of the lemma.

\end{proof}

\end{document}

%% file: macros.tex
\usepackage{bm,bbm}
\usepackage{algorithmic}
\usepackage{graphicx}
\usepackage{textcomp}
\usepackage{paralist}
\usepackage{xcolor}
\usepackage{thmtools}

\usepackage[utf8]{inputenc}

\usepackage{mathtools}
\usepackage{csquotes}
\usepackage[inline]{enumitem}
\usepackage{nicefrac}
\usepackage[capitalize,noabbrev]{cleveref}
\usepackage{amsthm}

\usepackage{xcolor}

\DeclareMathOperator*\Mean{\mathrm{Avg}}
\DeclareMathOperator*\LMean{\mathrm{LAvg}}

\newcommand{\supl}{\mathrm{Max}}
\newcommand{\infl}{\mathrm{Min}}

\newcommand\Aggo[1]{\mathrm{Agg}{[}#1]}
\newcommand{\Agglang}{\mathrm{Agg}}

\newcommand{\whp}{\mathrm{w.h.p.}}

\newcommand{\atomictypes}{\mathcal{A}}

\newcommand{\atppercent}{\%}

\newcommand{\closure}{\mathrm{cl}}
\newcommand{\atomictypeof}{\mathrm{Atype}}
\newcommand{\closuretype}{\mathcal{C}}
\newcommand{\closuretypes}{\mathcal{C}} 
\newcommand{\closuretypeof}{\mathrm{Ctype}}

\newcommand{\functionclass}{{\mathcal{F}}}

\newcommand{\functionclasspoly}{\functionclass_{\mathrm{rlpoly}}}

\newcommand{\functionclassrellip}{\functionclass_{\mathrm{rellip}}}

\newcommand{\Lipsch}{{\functionclass_{LP}}}

\newcommand{\R}{\ensuremath{\mathbb{R}}}
\newcommand{\reals}{\R}

\newcommand{\N}{\ensuremath{\mathbb{N}}}
\newcommand{\nats}{\N}

\newtheorem{theorem}{Theorem}[section]
\newtheorem{fact}[theorem]{Fact}
\newtheorem{claim}[theorem]{Claim}

\newtheorem{corollary}[theorem]{Corollary}
\newtheorem{lemma}[theorem]{Lemma}

\theoremstyle{remark}
\newtheorem{remark}[theorem]{Remark}

\theoremstyle{definition}
\newtheorem{definition}[theorem]{Definition}

\DeclareMathOperator\Ne{\mathcal N}

\let\Pr\undefined
\DeclareMathOperator*\Pr{\mathbb{P}}

\newcommand\lra\leftrightarrow
\newcommand\Lra\Leftrightarrow
\newcommand\dsb\doublesqbracket

\DeclarePairedDelimiter{\doublesqbracket}{\llbracket}{\rrbracket}
\DeclarePairedDelimiter{\abs}{|}{|}

\def\erdosrenyi{Erd\H{o}s--R\'enyi}

\newcommand{\NN}{\mathbb{N}}
\newcommand{\GG}{\mathbb{G}}

%% file: abstract.tex
\begin{abstract} Concentration results say that a sequence of random variables becomes progressively concentrated around the mean. Such results are  common in the study
of functions of random graphs.  We introduce a real-valued logic with various aggregate operators on graphs,  including summation, and prove that every term in the language, seen as a random variable on random graphs within the classical {\erdosrenyi} random graph model, is concentrated.  We prove this for dense and sparse variants of {\erdosrenyi} graphs. On the one hand, our results extend the line of work originating with Fagin and Glebskii et al. on zero-one laws
for dense random graphs, as well as the zero-one law of Shelah and Spencer for sparse random graphs.  On the other hand, they can be seen as a meta-theorem for inferring
concentration results on random graphs, and we give examples of such applications.
\end{abstract}